\documentclass{article}

\usepackage[utf8]{inputenc}
\usepackage{amsmath,amsfonts,amssymb}
\usepackage{amsthm,upref}
\usepackage[a4paper]{geometry}
\usepackage{color} 
\usepackage[dvipsnames]{xcolor}
\usepackage{graphicx}
\usepackage{subfigure}
\usepackage{tabularx,float}
\usepackage[hidelinks]{hyperref}
\usepackage{comment} 
\usepackage{tabularx} 
\usepackage{titlefoot} 
\usepackage{fancyhdr}

\newtheorem{thm}{Theorem}[section]
\newtheorem{definition}[thm]{Definition}

\newtheorem{lemma}[thm]{Lemma}
\newtheorem{remark}[thm]{Remark}
\newtheorem{assumption}{Assumption}

\usepackage[dvipsnames]{xcolor}
\usepackage{graphicx}
\usepackage{bm}

\newcommand{\dd}{\textrm{d}}
\newcommand{\Diff}{\textup{D}}
\newcommand{\Lop}{\textup{L}}
\newcommand{\N}{\textup{N}}
\newcommand{\M}{\textup{M}}
\newcommand{\eps}{\varepsilon}
\newcommand{\R}{\mathbb{R}}

\newcommand{\real}{\textrm{Re}\,}
\newcommand{\imag}{\textrm{Im}\,}

\newcommand{\transpose}{\textnormal{T}}
\newcommand{\me}{\mathrm{e}}

	\title{A Formal Geometric Blow-up Method\\for Pattern Forming Systems}
	
	\author{S.~Jelbart\footnote{Corresponding author. Email: jelbart@ma.tum.de}$^{,\dagger}$ \ \& C.~Kuehn\footnote{Department of Mathematics, Technical University of Munich, Garching near Munich, Bavaria 85748, Germany. SJ and CK acknowledge funding from the SFB/TRR 109 Discretization and Geometry in Dynamics grant. CK was additionally supported by Lichtenberg Professorship of the VolkswagenStiftung.}}
	
\begin{document}
	
	\maketitle
	
	\begin{abstract}
		We extend and apply a recently developed approach to the study of \textit{dynamic bifurcations} in PDEs based on the geometric blow-up method. We show that this approach, which has so far only been applied to study a dynamic Turing bifurcation in a cubic Swift-Hohenberg equation, can be coupled with a fast-slow extension of the \textit{method of multiple scales}. This leads to a formal but systematic method, which can be viewed as a fast-slow generalisation of the formal part of classical modulation theory. We demonstrate the utility and versatility of this method by using it to derive \textit{modulation equations}, i.e.~simpler closed form equations which govern the dynamics of the formal approximations near the underlying bifurcation point, in the context of model equations with dynamic bifurcations of (i) Turing, (ii) Hopf, (iii) Turing-Hopf, and (iv) stationary long-wave type. The modulation equations have a familiar form: They are of real Ginzburg-Landau (GL), complex GL, coupled complex GL and Cahn-Hilliard type respectively. In contrast to the modulation equations derived in classical modulation theory, however, they have time-dependent coefficients induced by the slow parameter drift, they depend on spatial and temporal scales which scale in a dependent and non-trivial way, and the geometry of the space in which they are posed is non-trivial due to the blow-up transformation. The formal derivation of the modulation equations provides the first steps toward the rigorous treatment of these challenging problems, which remains for future work.
		%
	\end{abstract}

	\bigskip

	\noindent {\small \textbf{Keywords:} Geometric blow-up, Modulation theory, Amplitude equations, Singular perturbation theory, Ginzburg-Landau equation.}
	
	\noindent {\small \textbf{MSC2020:} Primary 35B25, 35B32, 35B36; Secondary 37L99, 37G10}
	
	\section{Introduction}
	\label{sec:introduction}
	
	The goal of this article is to demonstrate the utility and versatility of the geometric blow-up method as a means for studying dynamic bifurcations in PDEs. We shall focus on the description of `essentially PDE' phenomena, with a particular emphasis on bifurcations associated to the emergence of spatial, temporal and spatio-temporal patterns in model problems which are obtained as `fast-slow extensions' of static bifurcation problems where the control parameter is allowed to vary slowly in time.
	
	From a mathematical point of view, all of the model problems we consider can be written in the general form
	\begin{equation}
		\label{eq:general_pde_intro}
		\partial_t \bm u = \left(\M + \Lop \right) \bm u + \N(\bm u, \mu, \eps) , \qquad
		\dot \mu = \eps ,
	\end{equation}
	where $t \geq 0$, $\bm u = \bm u(\bm x,t) \in \R^N$, $\M$ is a real-valued $N \times N$ constant coefficient matrix, $\Lop$ is a differential operator, $\N$ is a (sufficiently smooth) nonlinear operator, $\mu = \mu(t) \in \R$, $\dot \mu = \dd \mu / \dd t$ and $\eps \in [0,\eps_0]$ is a small perturbation parameter. We assume that the spatial domain $\Omega \subseteq \R^n$ is unbounded in at least one direction, which is a common assumption in pattern forming problems for which the characteristic (spatial, temporal or spatio-temporal) wavelength of the pattern is small in comparison to the length scale of the domain. System \eqref{eq:general_pde_intro} can be viewed as a `fast-slow extension' of a parameter-dependent PDE because the slow variable $\mu(t)$ is frozen in the limiting problem \eqref{eq:general_pde_intro}$|_{\eps = 0}$, referred to herein as the \textit{static subsystem}, and can therefore be treated as a control parameter. We are primarily interested in systems \eqref{eq:general_pde_intro} for which the static subsystem undergoes a bifurcation as $\mu$ is varied over some critical value $\mu_c$, in which case we say that the fast-slow system \eqref{eq:general_pde_intro} with $0 < \eps \ll 1$ undergoes a \textit{dynamic bifurcation}. We shall consider dynamic bifurcations in model problems with and underlying static bifurcation of Turing, Hopf, Turing-Hopf and long wave type. These are in direct correspondence with the four generic possibilities which appear in the linear classification of Cross \& Hohenberg \cite{Cross1993}; see also \cite{Frohoff2023} for a recent classification which also takes the presence or absence of conservation laws into account.
	
	One of the main analytical challenges in the static setting, i.e.~with regard to system \eqref{eq:general_pde_intro}$|_{\eps = 0}$, stems from the unbounded spatial domain. For the applications of interest here, this implies a crossing of \textit{continuous} spectra into the right-half plane as $\mu$ is varied over $\mu_c$. As a consequence, classical reductions to either finite or infinite ODE systems based on center manifold theory \cite{Haragus2010,Vanderbauwhede1992} are not directly applicable. In this case, the instabilities are characterised by a continuous band of unstable modes. An established alternative is to use \textit{modulation theory}, which can be seen as an attempt to generalise a large part of the approach in center manifold theory to the case of systems with continuous spectra \cite{Schneider2017}. The overall structure of a typical argument in modulation theory is as follows:
	\begin{itemize}
		\item[(I)] Introduce a multi-scale ansatz based on spectral information in the unstable but weakly nonlinear regime $|\mu - \mu_c| \ll 1$;
		\item[(II)] Derive a simpler \textit{modulation equation} using formal asymptotic arguments;
		\item[(III)] Show that solutions to the original problem are attracted to and well approximated by solutions to the modulation equation.
	\end{itemize}
	Step (I) is motivated by the expectation that the dynamics close to the onset of instability (i.e.~close to the bifurcation point) should be primarily driven by the interaction of modes in an asymptotically small (but still uncountably infinite) subset of modes contained in a ball around the unstable band(s). The modulation equation derived using formal methods in Step (II) will still have uncountably many degrees of freedom, and will therefore in general be a PDE, but it should in some sense be `simpler' given that the vast majority of modes have been pushed into higher orders by the structure of the multi-scale ansatz from Step (I). The aim in Step (III) is to show that the modulation equation plays a similar role to the reduced equations on a center manifold. In particular, the well-known approximation and attractivity properties of center manifolds are necessary for the vast majority of applications. These properties need to be checked explicitly in order to `justify' or `validate' Steps (I)-(II), and this has been done for a wide range of applications; see \cite{Schneider2017} and the many references therein. Although we shall focus primarily on Steps (I)-(II) in the fast-slow setting in this work, it is worthy to emphasise the importance of Step (III), particularly given that a growing number of case studies show that Step (III) can fail even if the modulation equation in Step (II) is derived in a formally correct way \cite{Schneider1995a,Schneider2005,Schneider2015}. 
	
	\
	
	The primary contribution of the present article is to generalise Steps (I)-(II) for fast-slow problems in the general form \eqref{eq:general_pde_intro} (although we also expect the methods to apply outside of this class). In order to generalise Step (I), we build upon recent work in \cite{Jelbart2022} where it was shown for a dynamic Turing instability in a Swift-Hohenberg equation that the relevant multi-scale ansatz can be formulated as a geometric blow-up transformation. In this approach, a dynamic generalisation of the classical ansatz is obtained by replacing the usual small parameter with a time-dependent variable $r(t)$ which measures the distance of solutions from the so-called \textit{blow-up manifold}, which replaces the static bifurcation point in an auxiliary phase space referred to as the \textit{blown-up space}. This space tends to have a more complicated geometry, but improved dynamical properties which allow for a more effective application of certain linearisation techniques. An important feature of this approach is that the simple rescalings of time and space have to be replaced by more complicated relationships which allow for the ratio of different spatial and temporal scales to change as a function of $r(t)$. This is similar to the use of time-dependent transformations of space and time in \textit{dynamic renormalization} approaches \cite{Bricmont1995,Chapman2021,Siettos2003}. 
	The approach developed in \cite{Jelbart2022} and this article is novel for PDEs, but its feasibility is supported to a certain extent by a long history of successes with regard to the study of dynamic bifurcations in finite-dimensional fast-slow systems using the geometric blow-up method; see \cite{Dumortier1996,Krupa2001a,Krupa2001c,Krupa2001b,Szmolyan2001,Szmolyan2004} for important early works and \cite{Jardon2019b} for a recent survey.
	
	In order to generalise Step (II), we will show that the geometric blow-up approach proposed in Step (I) can be coupled to the classical formal asymptotic approach known as the \textit{method of multiple scales} \cite{Kevorkian2012,Kuramoto1984,Nayfeh2008}. The main task is to show that a formally correct generalisation of this method to the fast-slow setting is possible. In order to do so, we need to account for a number of complications including (i) the fact that the usual perturbation parameter has been replaced by a time-dependent variable, (ii) the fact that the simple rescalings of space and time are more complicated due to (i), and (iii) the fact that we work in a non-trivial geometry (the blown-up space). Although each of these complications may be expected to lead to (potentially significant) complications in applications, we will see that structurally, they do not pose a significant obstacle to a systematic extension to the fast-slow setting.
	
	\
	
	We will focus on the 
	general system \eqref{eq:general_pde_intro}, and four different model problems within this class. The `output' in each case, is a (generally non-autonomous) modulation equation which describes the evolution of the formal leading order approximation in the blown-up space. The model problems are obtained as fast-slow extensions of the following:
	\begin{itemize}
		\item[(M1)] A cubic Swift-Hohenberg equation with a Turing instability;
		\item[(M2)] A two-component Brusselator (reaction-diffusion) system with a Hopf bifurcation;
		\item[(M3)] A system of coupled Kuramoto-Sivashinsky equations with a Turing-Hopf bifurcation;
		\item[(M4)] A two-component, two-dimensional Navier-Stokes problem with a conservative long wave bifurcation.
	\end{itemize}
	In each case, we rely on existing calculations and results for the static counterparts to these problems. See \cite{Collet1990,Eckhaus1993,Jelbart2022,Kirrmann1992} for (M1), \cite{Glansdorff1971,Schneider1998} for (M2), \cite{Schneider1997,Schneider2017} for (M3) and \cite{Meshalkin1961,Nepomniashchii1976,Schneider1999b} for (M4). Our methods allow for the derivation of a real Ginzburg-Landau (GL) equation in the first case, a complex GL equation in the second case, a system of generalised complex GL equations in the third case, and a Cahn-Hilliard equation in the fourth case. The general form of the modulation equation is in each case closely related to the corresponding equation in static modulation theory, except for a number of important distinctions. In the fast-slow setting, additional terms arise due to the parameter drift, the coefficients of the modulation equations can depend on time, the relationship between spatial and temporal scales is more complicated, and the equations are posed in the non-trivial geometry of the blow-up space.
	
	\
	
	There is a growing literature on dynamic bifurcations in PDEs, and there are many important works which have made significant contributions on problems which are similar to or closely related to the core topic of this article. For an alternative approach to geometric blow-up which has recently been developed for fast-slow PDEs on bounded domains we refer to \cite{Engel2022,Engel2020}. We refer to \cite{Avitabile2020} for rigorous results on a range of dynamic bifurcations which satisfy the spectral gap requirement of the center manifold theory in \cite{Haragus2010,Vanderbauwhede1992}. Rigorous results on the exchange of stability and/or delayed stability loss phenomenon have been derived using upper and lower solutions for scalar PDEs in \cite{Butuzov2002b,Butuzov2000,Butuzov2001,Butuzov2002,Kaper2018,Nefedov2003}, and we refer to \cite{Bilinsky2018,Goh2022,Kaper2018,Kaper2021} for detailed formal and numerical studies of delayed Hopf bifurcation. Finally, travelling wave dynamics induced by the slow passage through pitchfork and fold type bifurcations have recently been analysed in \cite{Goh2022}.
	
	
	
	
	\
	
	The remainder of the article is structured as follows: In Section \ref{sec:types_of_Turing_instablity} we introduce basic notions, assumptions and definitions for the four (static) bifurcations of interest. In Section \ref{sec:dynamic_Turing_instability} we introduce four particular model problems (one for each bifurcation type), show that they undergo a particular type of dynamic bifurcation, and present the known modulation approximation for the underlying static problem. In Section \ref{sec:geometric_blow-up_and_the_modulation_equations} we outline and apply the geometric blow-up approach for the general class of equations under consideration, and couple it to the method of multiple scales. In Section \ref{sec:model_problems} we apply this method in order to derive modulation equations for each of the model problems. We conclude with a summary and discussion of our findings in Section \ref{sec:conclusion_and_outlook}.

	\section{Types of Turing instability: Four bifurcations}
	\label{sec:types_of_Turing_instablity}
	
	In this section we introduce the class of (static) systems of interest, along with defining conditions for the four types of bifurcations for which we consider dynamic generalisations in later sections. These are in direct correspondence with the four basic bifurcation types distinguished in the linear classification of Cross \& Hohenberg \cite[Ch.~IIIb]{Cross1993}. 
	
	We shall be interested in dynamic extensions of a parameter-dependent PDEs of the form
	\begin{equation}
		\label{eq:general_pde}
		\partial_t \bm{u} = (\M + \Lop) 
		\bm{u} + \N(\bm{u}, \mu) , 
	\end{equation}
	where $t \geq 0$, $\bm u = \bm u(\bm x,t) \in \R^N$, $\M$ is an $N \times N$ matrix with constant coefficients in $\R$ and $\mu \in \R$ is a control parameter. The operator $\Lop = \textup{diag}\ (\Lop_1,\ldots,\Lop_N)$ is an order $m \in \mathbb N$ linear differential operator with either constant or $\bm x$-dependent coefficients, i.e.~the components are given by
	\begin{equation}
		\label{eq:L}
		\Lop_j
		= \sum_{|\alpha| \leq m} a_{j}^{(\alpha)}(\bm x) \Diff_{\bm x}^\alpha 
		\qquad
		j = 1,\ldots,N,
	\end{equation}
	where $\alpha = (\alpha_1,\ldots,\alpha_n)$ is a multi-index, $|\alpha| = \alpha_1 + \ldots + \alpha_n$, $a_{j}^{(\alpha)} \in C^\infty(\widetilde X, \R)$ for some open domain $\widetilde X$ in $\mathbb R^n$, and
	$\Diff_{\bm x}^\alpha :=
		\partial^{|\alpha|} /
		\partial {x_1}^{\alpha_1} \partial {x_2}^{\alpha_2} \dots \partial x_n^{\alpha_n}$.
	The operator
	$\N$, 
	which may also depend on $\bm x$, $\Diff_{\bm x} \bm u$ and higher spatial derivatives, is nonlinear in $(\bm u, \mu)$ and satisfies 
	\begin{equation}
		\label{eq:N_cond}
		\N(\bm 0, \mu) = \bm 0 
	\end{equation}
	for all $\mu \in I$ and $\bm x \in \Omega \subseteq \R^n$, where $I$ is a neighbourhood of $0$ in $\R$. We assume that the spatial domain $\Omega$ is unbounded in at least one direction. More precisely, $\bm x = (x_1, x_2, \ldots, x_n) \in \Omega = \R^p \times \widetilde \Omega \subseteq \R^n$ for $p \in \{1,\ldots,n\}$, where $\widetilde \Omega \subset \R^{n-p}$ is bounded if $p < n$ and empty if $p = n$. 
	
	\begin{remark}
		Since the calculations presented herein are primarily formal, we do not specify the phase space. Boundary conditions on $\widetilde \Omega$ are also not specified at this point. Periodic boundary conditions are assumed for the only model problems with $\widetilde \Omega \neq \emptyset$ considered in later sections.
	\end{remark}
	
	A number of the preceding assumptions, which imply the existence of a spatially homogeneous steady state $\bm u(\bm x,t) = \bm u^\ast(t) \equiv 0$, are summarised in the following.
	
	\begin{assumption}
		\label{ass:steady_state}
		$\bm u(\bm x,t) = \bm u^\ast(t) \equiv \bm 0$ defines a spatially homogeneous steady state for \eqref{eq:general_pde} for all $\mu \in I \subseteq \R$, where $I$ is a neighbourhood of zero. The eigenvalue problem obtained by linearisation about $\bm u^\ast$ is well-defined and given by
		\begin{equation}
			\label{eq:eigenvalue_problem}
			\left( \M + \widehat{\Lop} + \N'(\bm 0,\mu) \right) \widehat{\bm v}_j = \lambda_j(\bm \xi,\mu) \widehat{\bm v}_j ,
		\end{equation}
		where $\widehat{\Lop}$ is the diagonal matrix with entries $\sum_{|\alpha| \leq m} a_{j}^{(\alpha)}(\bm x) i^{|\alpha|} \xi^\alpha$, 
		$\bm \xi = (\bm \xi_{ub}, \bm \xi_b) \in \R^p \times \mathbb Z^{n-p}$ is the wave vector, $\xi^\alpha = \xi_1^{\alpha_1} \xi_2^{\alpha_2} \cdots \xi_n^{\alpha_n}$,
		$\N'(\bm 0,\mu)$ is the Frech\'et derivative/Jacobian of $\N$ with respect to $\bm u$ at $\bm u^\ast$, and $\widehat{\bm v}_j 
		\in \mathbb C^N$ is the eigenfunction. 
	\end{assumption}
	
	
	Assumption \ref{ass:steady_state} is not restrictive at all for the applications we have in mind: The eigenvalue problem \eqref{eq:eigenvalue_problem} can be obtained after Fourier transform in space, and it is well-defined for as long as $\N$ is sufficiently regular in $\bm u$ (it should be Frech\'et differentiable at $\bm u^\ast$ so that $\N'(\bm 0, \mu)$ exists). For the applications of interest in this work, the unbounded spatial domain $\Omega = \R^p \times \widetilde \Omega$ implies that the spectrum defined by \eqref{eq:eigenvalue_problem} is continuous.
	
	\begin{assumption}
		\label{ass:symmetry}
		Let
		\begin{equation}
			\label{eq:xi}
			\xi := 
			\begin{cases}
				\bm \xi_{ub} \in \R, & p = 1, \\
				|\bm \xi_{ub}| \in \R_{\geq 0}, & p \geq 2.
			\end{cases}
		\end{equation}
		The maps $\xi \mapsto \real \lambda_j(\xi,\mu)$ are well-defined for all $j \in \mathbb N_{\geq 0}$, and their graphs define $p$-dimensional surfaces which are ordered according to
		%
		$\real \lambda_j(\xi,\mu) \geq \real \lambda_{j+1}(\xi,\mu)$. 
	\end{assumption}
	
	
	Let $\bm x := (\bm x_{ub}, \bm x_b) \in \R^p \times \widetilde \Omega$, where the subscripts `ub' and `u' highlight that we have separated ``unbounded" and ``bounded" directions respectively. Assumption \ref{ass:symmetry} can be expected to hold under suitable boundary conditions on $\widetilde \Omega$ if the linear operator $\widehat \Lop + \N'(\bm 0,\mu)$, and therefore $\M + \widehat \Lop + \N'(\bm 0, \mu)$ in \eqref{eq:eigenvalue_problem}, is translation invariant in $\bm x_{ub}$, since this leads to a countable number of eigenvalue surfaces corresponding to eigenfunctions proportional to $\me^{i \bm \xi_{ub} \cdot \bm x_{ub} + \lambda_j t}$. We include it here in order to avoid having to make the additional restrictions on $\Lop$, $\N$ and the boundary conditions for $\widetilde \Omega$ explicit at this point.
	
	\
	
	It follows from Assumptions \ref{ass:steady_state}-\ref{ass:symmetry} that the steady state $\bm u^\ast$ undergoes an instability when the surface defined by $\xi \mapsto \real \lambda_1(\xi,\mu)$ crosses zero for some $\xi = \xi_c \in \R$ under variation in $\mu$ past some critical value $\mu_c \in I$. Note that this implies a crossing at $\xi = -\xi_c$, too, since \eqref{eq:general_pde} is real-valued, which implies that $\lambda_j(-\xi,\mu) = \overline{\lambda_j}(\xi,\mu)$ and therefore $\real \lambda_j(\xi,\mu) = \real \lambda_j(-\xi,\mu)$. Moreover, equation \eqref{eq:xi} forces $\xi_c = 0$ if $p \geq 2$; see Remark \ref{rem:symmetry} below. For simplicity, we assume that such an instability occurs for $\mu_c = 0$.
	
	\begin{assumption}
		\label{ass:bifurcation}
		The steady state $\bm u^\ast$ undergoes an instability as $\mu$ is varied over $\mu_c = 0$. Specifically, we assume that the following conditions are satisfied: 
		\begin{itemize}
			\item[(i)] $\mu < 0 \ \implies \ \real \lambda_1(\xi,\mu) < 0$ for 
			all $\xi \in \R$;
			\item[(ii)] $\mu = 0 \ \implies \ \real \lambda_1(\pm \xi_c, 0) = 0$ and $\real \lambda_1(\xi,0) < 0$ for all $\xi \neq \xi_c$;
			\item[(iii)] $\mu > 0 \ \implies \ \real \lambda_1(\xi,\mu) > 0$ for all $\xi$ such that $|\xi| \in (\xi_-(\delta), \xi_+(\delta))$, where $\xi_\pm(\delta)$ are continuous functions depending on $\mu =: \delta^2 \ll 1$ satisfying $\xi_-(0) = \xi_+(0)$. Moreover, $\real \lambda_1(\xi_\pm(\delta),\mu) = 0$ and $\real \lambda_1(\xi,\mu) < 0$ for all $\xi$  such that $|\xi| \notin [\xi_-(\delta), \xi_+(\delta)]$.
		\end{itemize}
	\end{assumption}
	
	Assumption \ref{ass:bifurcation} implies that $\bm u^\ast$ becomes unstable as $\mu$ is varied over zero. Since the spectrum is continuous, there is an entire band of unstable modes with wavenumbers $\xi \in (\xi_-(\delta), \xi_+(\delta))$ when $\mu > 0$. On the linear level, there are four basic types of instability, depending on whether or not $\xi_c$ and/or $\omega_c := \imag \lambda_1(\xi_c,0)$ are zero; we refer again to \cite[Ch.~IIIb]{Cross1993} for details.
	
	\begin{figure}[t!]
		\centering
		\includegraphics[width=0.45\textwidth]{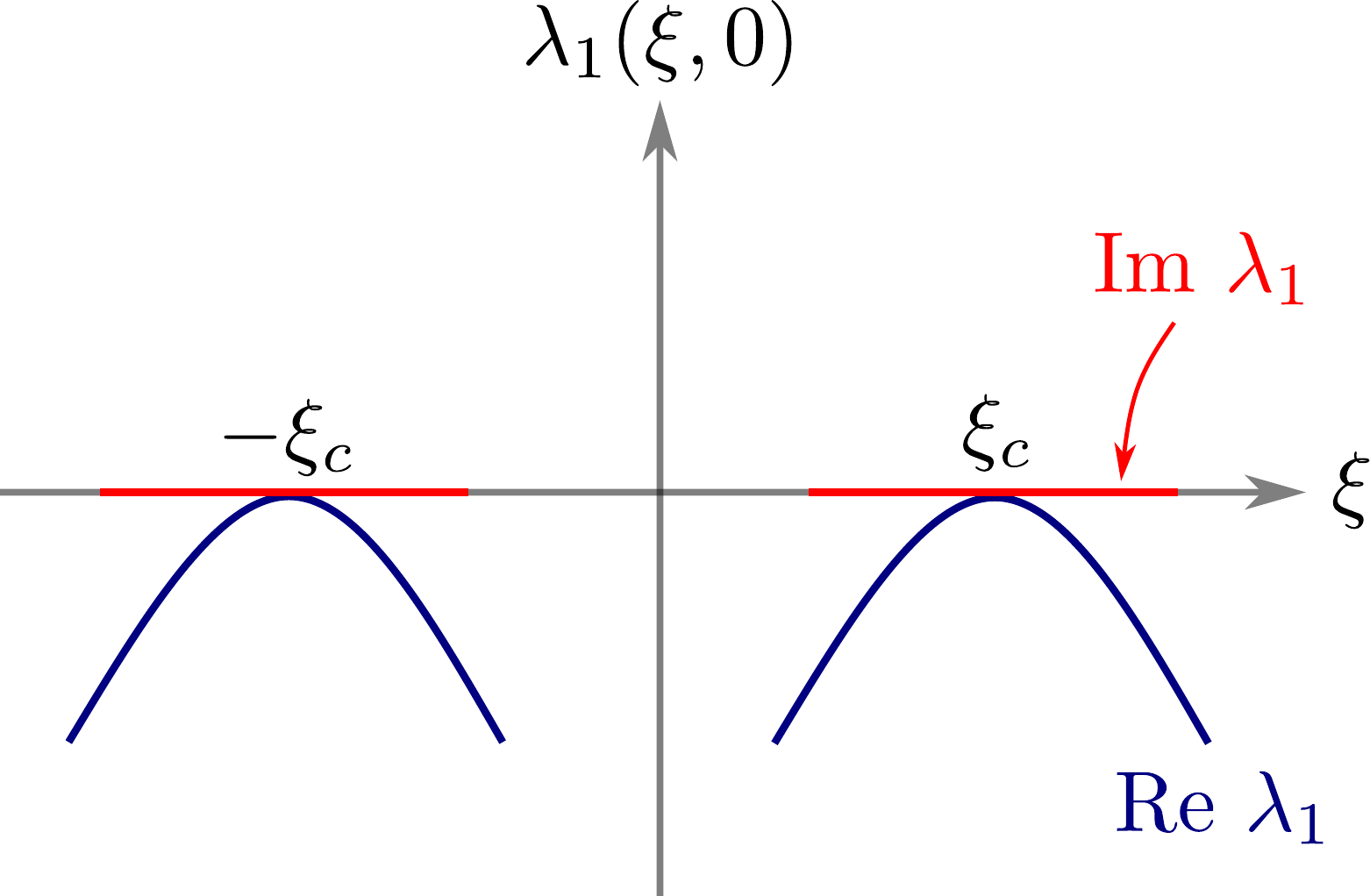} \ \ \ \ 
		\includegraphics[width=0.45\textwidth]{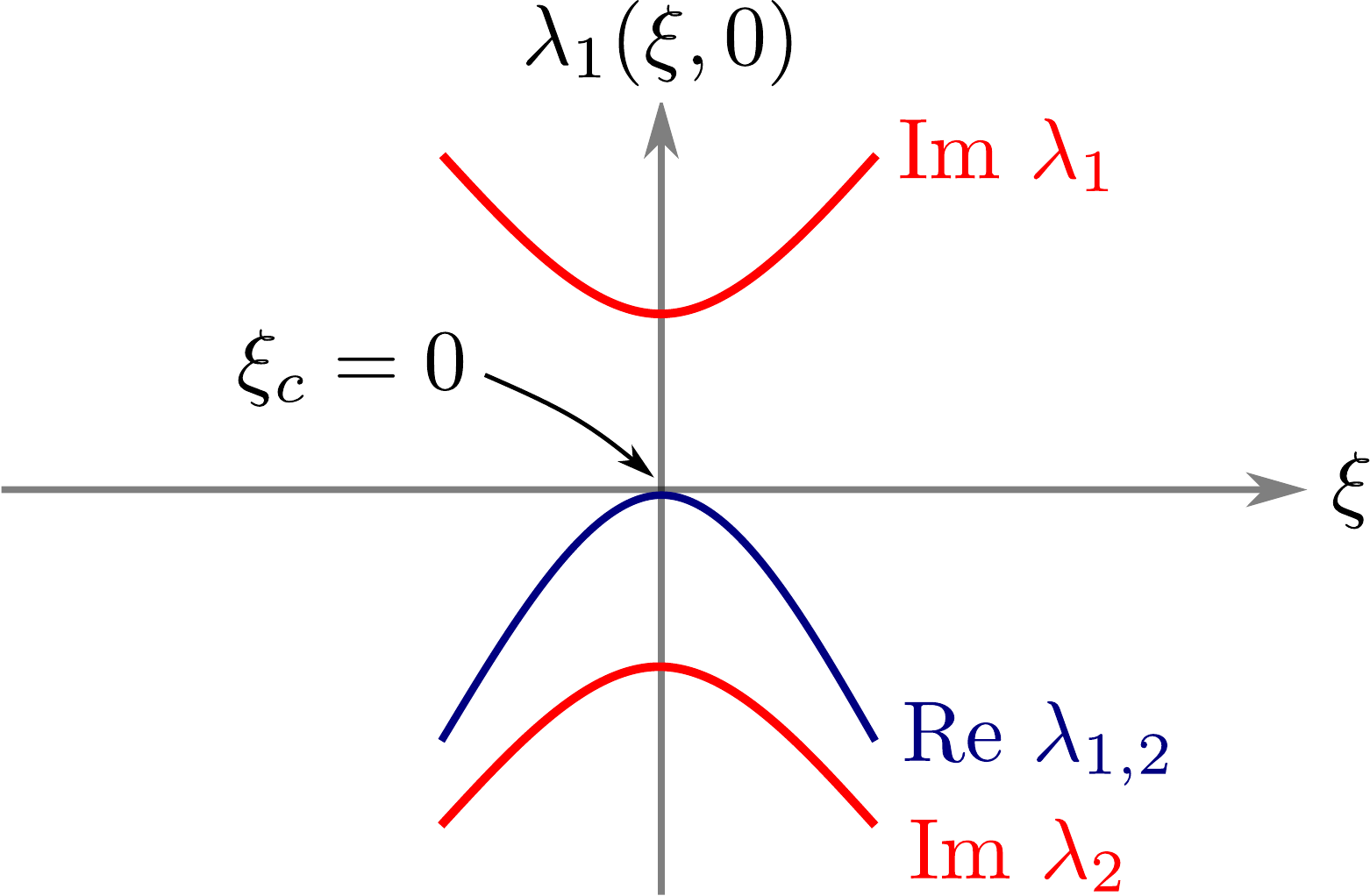}
		
		\
		
		\includegraphics[width=0.45\textwidth]{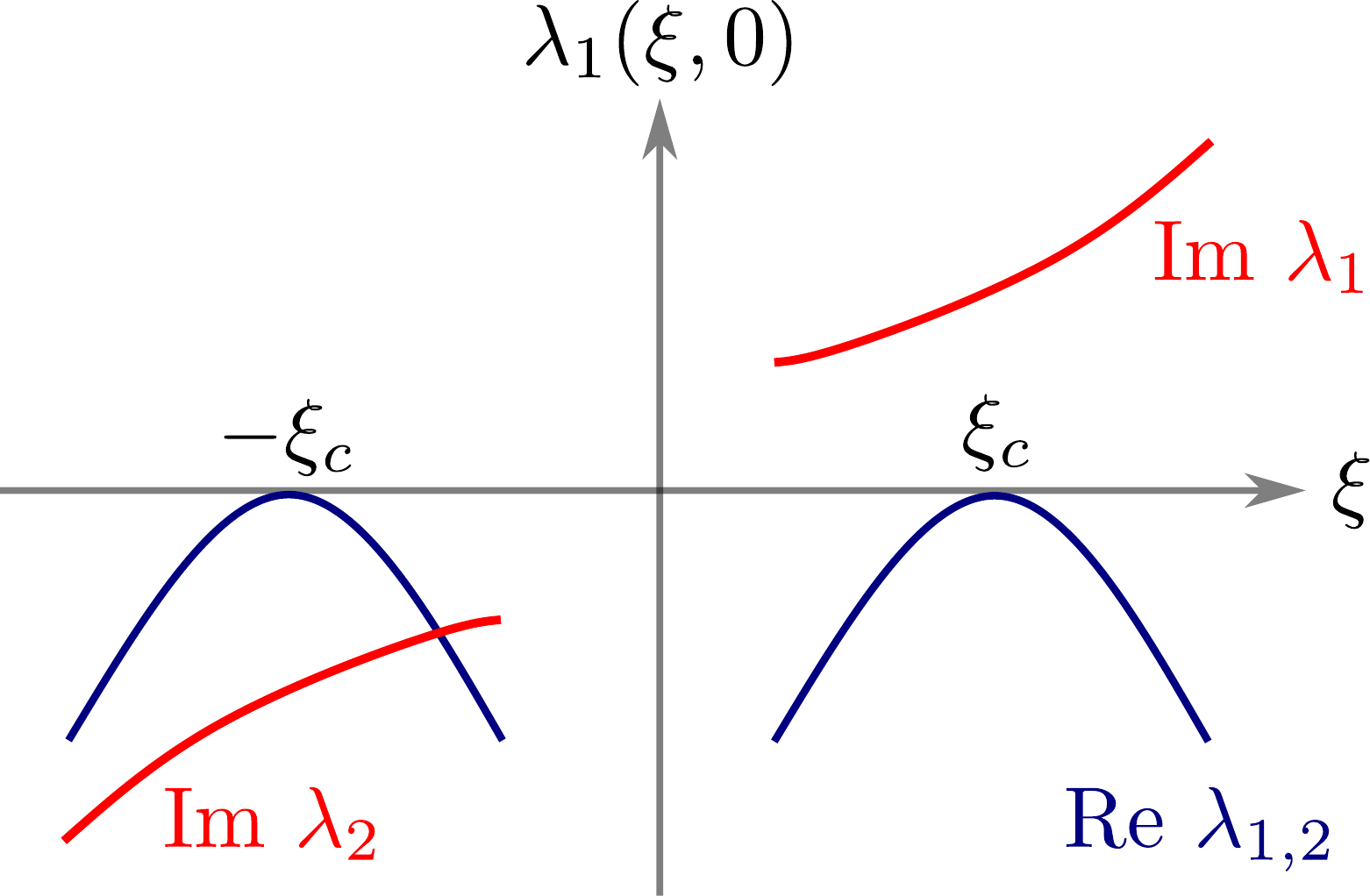} \ \ \ \ 
		\includegraphics[width=0.45\textwidth]{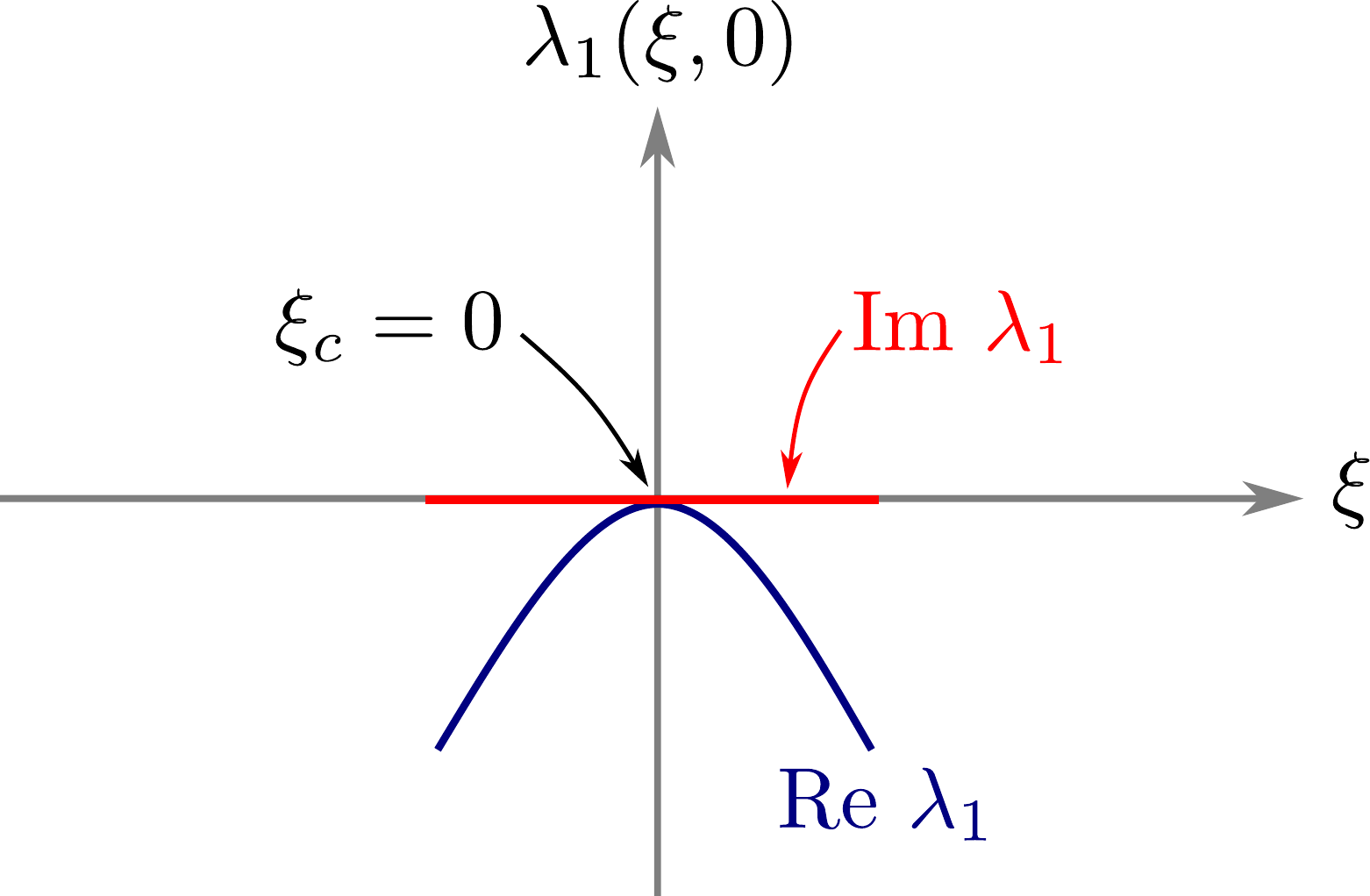}
		\caption{Real and imaginary parts of the leading eigenvalue $\lambda_1(\xi,0)$ at a bifurcation of Turing type (top left), Hopf type (top right), Turing-Hopf type (bottom left) and long-wave type (bottom right); cf.~Definition \ref{def:bifurcations_static}. For visual purposes, we sketch the situation for the case $x \in \Omega = \R$. Note that $\lambda_1(-\xi,0) = \overline{\lambda_1}(\xi,0)$ because \eqref{eq:general_pde} is real-valued.}
		\label{fig:spectra}
	\end{figure}
	
	\begin{definition}
		\label{def:bifurcations_static}
		Consider the PDE \eqref{eq:general_pde} under Assumptions \ref{ass:steady_state}-\ref{ass:bifurcation}. We say that there is a bifurcation with critical modes $\xi = \pm \xi_c$ at $\mu = 0$ of
		\begin{enumerate}
			\item Turing type if $\xi_c \neq 0$ and $\omega_c = 0$; 
			\item Hopf type if $\xi_c = 0$ and $\omega_c \neq 0$; 
			\item Turing-Hopf type if $\xi_c \neq 0$ and $\omega_c \neq 0$; 
			\item Long-wave type if $\xi_c = 0$ and $\omega_c = 0$. 
		\end{enumerate}
	\end{definition}
	
	A representative spectrum for each bifurcation is shown in Figure \ref{fig:spectra}. 
	Turing, Hopf and Turing-Hopf type bifurcations play an important role in facilitating the onset of spatial, temporal and spatio-temporal patterns respectively. Long-wave bifurcations are not directly associated with pattern formation, since (on the linear level) there is no oscillatory component in either space or time. We refer to the well-known references \cite{Cross1993,Hoyle2006,Schneider2017} for additional details, since we are not focused on pattern formation per se in this work.
	
	
	\begin{remark}
		\label{rem:symmetry}
		Assumption \ref{ass:symmetry} restricts us to spatial domains with one unbounded direction for Turing and Turing-Hopf type bifurcations, since $\xi_c \neq 0$ implies  that $-\xi_c < 0$, which is only possible for $p = 1$ due to \eqref{eq:xi}. For bifurcations of Hopf and long-wave type we may take $p > 1$, in which case the graph of $\real \lambda_1$ is a $p$-dimensional rotationally symmetric version of the corresponding curves shown in Figure \ref{fig:spectra}. 
	\end{remark}
	
	\begin{remark}
		\label{rem:center_manifold_theory}
		The presence of a continuous band of unstable modes prohibits the reduction to a finite or countably infinite system of ODEs. In particular, classical center manifold approaches such as those developed and applied in e.g.~\cite{Avitabile2020,Haragus2010,Vanderbauwhede1992}, which rely on a spectral gap property, do not apply. See however \cite{Roberts2015,Roberts2022} for a center manifold based approach to the study of (static and dynamic) bifurcations without a spectral gap which utilises a spectral gap in a larger ensemble system into which the original problem is embedded.
	\end{remark}
	
	\begin{remark}
		\label{rem:spatial_dimensions}
		In the preceding discussion we restricted to the case that $\mu \in \R$. Multi-parameter extensions to the case $\mu \in \R^l$ with $l > 1$ are straightforward, given that we consider codimension-1 bifurcations.
	\end{remark}

	\section{Dynamic Turing instability: Definitions and model problems}
	\label{sec:dynamic_Turing_instability}
	
	We turn now to \textit{dynamic} bifurcation problems, which may be obtained as simple fast-slow extensions of the static bifurcation problems described in Section \ref{sec:types_of_Turing_instablity}. More precisely, we consider systems of the form
	\begin{equation}
		\label{eq:dynamic_general_form}
		\partial_t \bm u = \left(\M + \Lop \right) \bm u + \N(\bm u, \mu, \eps) , \qquad
		\dot \mu = \eps ,
	\end{equation}	
	where $t \geq 0$, $\bm u = \bm u(\bm x,t) \in \R^N$, $\mu = \mu(t) \in \R$, $\dot \mu = \dd \mu / \dd t$ and $\eps \in [0,\eps_0]$ is a small perturbation parameter. The operators $\M$, $\Lop$ and $\N$ are defined similarly to Section \ref{sec:types_of_Turing_instablity}, 
	except that $\N$ is also allowed to depend continuously on $\eps$. We also require that
	\begin{equation}
		\label{eq:N_cond_eps}
		\N(\bm 0, \mu, 0) = \bm 0 ,
	\end{equation}
	for all $\mu \in I$; cf.~\eqref{eq:N_cond}. As before, we assume that $\bm x \in \Omega = \R^p \times \widetilde \Omega \subseteq \R^d$ where $p \geq 1$ and $\widetilde \Omega$ is bounded (or empty), i.e.~there is at least one unbounded direction.

	\begin{definition}
		\label{def:static_subsystem}
		The limiting problem obtained after taking $\eps \to 0$, i.e.
		\begin{equation}
			\label{eq:static_general_form}
			\partial_t \bm u = \left(\M + \Lop \right) \bm u + \N(\bm u,\mu,0) ,
		\end{equation}
		where $\mu = \mu(0) = const.$, is called the \textit{static subsystem}.
	\end{definition}
	
	The `static' terminology reflects the fact that the slow variable $\mu(t) \in \R$ in \eqref{eq:dynamic_general_form} is frozen in \eqref{eq:static_general_form}, and considered as a control parameter. Note that equation \eqref{eq:static_general_form} is in the general form \eqref{eq:general_pde}. We are interested in situations where the static subsystem undergoes one of the four bifurcations described in Section \ref{sec:types_of_Turing_instablity}. 
	
	\begin{definition}
		\label{def:bifurcations_dynamic}
		Consider system \eqref{eq:dynamic_general_form} and assume that the static subsystem \eqref{eq:static_general_form} satisfies Assumptions \ref{ass:steady_state}-\ref{ass:bifurcation}. We say that system \eqref{eq:dynamic_general_form} has a \textit{dynamic bifurcation} of Turing, Hopf, Turing-Hopf, or long-wave type with critical modes $\xi = \pm \xi_c$ at $\mu = \mu_c$ if the static subsystem has a bifurcation of the corresponding type.
	\end{definition}
	
	In the following we introduce and identify dynamic bifurcations of each type in specific model problems, which we refer to as (M1)-(M4) as in Section \ref{sec:introduction}.

	\subsection{(M1): Dynamic Turing bifurcation}
	\label{sub:case_i_dynamic_proper_Turing_bifurcation}
	
	Dynamic Turing bifurcation in a Swift-Hohenberg equation with slow parameter drift been studied using geometric blow-up in \cite{Jelbart2022}. In its simplest form, the problem is given by
	\begin{equation}
		\label{eq:dynamic_sh}
			\partial_t u = - (1 + \partial_x^2)^2 u + \mu u - u^3 , \qquad
			\dot \mu = \eps ,
	\end{equation}
	with $u(x,t) \in \R$ and $x \in \Omega = \R$.
	
	
	System \eqref{eq:dynamic_sh} is already in the form \eqref{eq:dynamic_general_form}, with
	\[
	\M = 0, \qquad 
	\Lop = - (1 + \partial_x^2)^2 , \qquad 
	\N(u,\mu,\eps) = \mu u - u^3 .
	\]
	The static subsystem is the well-known cubic Swift-Hohenberg equation
	\begin{equation}
		\label{eq:static_sh}
		\partial_t u = - (1 + \partial_x^2)^2 u + \mu u - u^3 ,
	\end{equation}
	with control parameter $\mu \in \R$.
	Direct calculations show that equation \eqref{eq:static_sh} satisfies Assumption \ref{ass:steady_state}. In particular, the eigenvalue problem \eqref{eq:eigenvalue_problem} is satisfied with
	\begin{equation}
		\label{eq:spectrum_sh}
		\lambda(\xi, \mu) = -(1-\xi^2)^2 + \mu, \qquad \xi \in \R,
	\end{equation}
	and eigenfunction $\hat v = \me^{i x}$. Equation \eqref{eq:spectrum_sh} can be used to verify Assumptions \ref{ass:symmetry}-\ref{ass:bifurcation} directly. Assumption \ref{ass:symmetry} is immediate, as are the properties (i)-(ii) in Assumption \ref{ass:bifurcation} (here $\xi_c = 1$). Asymptotic expressions for $\xi_\pm(\delta)$ can be obtained by solving the equation $\lambda(\xi_\pm(\delta),\delta^2) = 0$, which yields
	\begin{equation}
		\label{eq:xi_pm_sh}
		\xi_\pm(\delta) = 1 \pm \frac{1}{2} \delta + O(\delta^2) , \qquad \text{as $\delta\rightarrow 0$.}
	\end{equation}
	Thus Assumption \ref{ass:bifurcation} (iii) is satisfied with $\xi_c = 1 \neq 0$ and $\omega_c = 0$. 
	Thus by Definition \ref{def:bifurcations_static}, the static subsystem \eqref{eq:static_sh} has a Turing bifurcation with critical modes $\xi = \pm \xi_c = \pm 1$ at $\mu = 0$. Correspondingly, by Definition \ref{def:bifurcations_dynamic}, system \eqref{eq:dynamic_sh} has a dynamic Turing bifurcation for these values.
	
	
	\
	
	The dynamics of the static subsystem \eqref{eq:static_sh} with $0 < \mu = \delta^2 \ll 1$ are known to be well-approximated for time-scales $t \sim \delta^{-2}$ by multi-scale solutions of the form
	\begin{equation}
		\label{eq:approximation_sh}
		u(x,t) = \delta \psi(x,X,T) = \delta \left(A(X,T) \me^{ix} + c.c.\right) + O(\delta^2) ,
	\end{equation}
	where $X = \delta x$, $T = \delta^2 t$, and $A(X,T) \in \mathbb C$ obeys the real GL equation
	\begin{equation}
		\label{eq:modulation_eqn_static_sh}
		\partial_T A = 4 \partial_X^2 A + A - 3 A |A|^2 .
	\end{equation}
	There are many results on the validity of the real GL equation \eqref{eq:modulation_eqn_static_sh} as a modulation equation for the static Swift-Hohenberg problem \eqref{eq:static_sh}. The most important of these are the approximation results in \cite{Collet1990,Kirrmann1992,VanHarten1991}, which ensure the closeness of solutions to \eqref{eq:static_sh} and \eqref{eq:modulation_eqn_static_sh} over a specified time interval, and the attractivity results in \cite{Eckhaus1993,Mielke1995,Schneider1995b}, which ensure that all solutions with initial conditions of size $O(\delta)$ are approximated by GL solutions after transient dynamics. Taken together, these results can be seen as replacing the usual approximation and attractivity properties of a center manifold. In fact, the approximation and attractivity properties have been proposed as criteria for a modulation equations to be viewed as generalisation of the situation in which solutions are approximated by a reduced system on a center manifold, to the case of problems with continuous spectrum \cite{Schneider1998,Schneider1999}.
	
	\begin{remark}
		For $p=1$, the real GL equation
		$
		\partial_T A = c_1 \partial_X^2 A + c_2 A + c_3 A |A|^2 
		$
		is as a universal modulation equation in the weakly nonlinear regime $0 < \mu \ll 1$ associated with a Turing instability at $\mu = 0$. In other words, it can be derived under generic assumptions, and its particular form in applications varies only via the values taken by the coefficients $c_1, c_2, c_3 \in \R$. If $c_3 < 0$, the GL equation can be used to prove the global existence of small (with respect to $L^\infty$) solutions in the `original problem' as $\mu \to 0^+$ \cite{Schneider1994b}, and the upper-semicontinuity of the original system attractor towards the GL attractor was shown in \cite{Mielke1995}. See \cite{Schneider2017} for details and further references.
	\end{remark}
	
	\begin{remark}
		\label{rem:scaling}
		For the modulation approximation \eqref{eq:approximation_sh} and all other model problems considered in this work, the relevant rescalings have the form $X = \delta^\gamma x$ and $T = \delta^\beta t$, for real exponents $\gamma, \beta > 0$. Values for $\gamma$ and $\beta$ can be directly inferred from the shape of the spectrum in the weakly nonlinear regime $0 < \mu \ll 1$. Specifically, $\delta^\gamma$ is the length scale of the band(s) of unstable modes $\xi$, and $\delta^\beta$ is the asymptotic order of magnitude of $\real \lambda_1(\xi_c,\mu)$. More precisely, $\gamma, \beta > 0$ are determined via the relations $| \xi_+(\delta) - \xi_-(\delta) | = O(\delta^\gamma)$ and $\real \lambda_1(\xi_c,\delta^2) = O(\delta^\beta)$.
	\end{remark}
	

	\subsection{(M2): Dynamic Hopf bifurcation}
	\label{sub:case_iv_dynamic_Hopf_bifurcation}
	
	We consider the slow passage through a Hopf bifurcation (a.k.a.~dynamic Hopf bifurcation) in a dynamic Brusselator system
	\begin{equation}
		\label{eq:dynamic_b_1}
		\begin{split}
			\partial_t \tilde u &= d_1 \Delta \tilde u + a - (b + 1) \tilde u + \tilde u^2 , \\
			\partial_t \tilde v &= d_2 \Delta \tilde v + b \tilde u - \tilde u^2 \tilde v ,  \\
			\dot b &= \eps ,
		\end{split}
	\end{equation}
	where $\bm x = (x_1,\ldots,x_n) \in \Omega = \R^n$, $\bm u(\bm x,t) = (\tilde u(\bm x,t), \tilde v(\bm x,t)) \in \R^2$, $b(t) \in \R$, $\Delta = \sum_{j=1}^n \partial_{x_j}^2$ and $d_1, d_2, a > 0$ are fixed parameters. Note that we do not restrict the spatial dimension $n$. The static subsystem
	\begin{equation}
		\label{eq:static_b_1}
		\begin{split}
			\partial_t \tilde u &= d_1 \Delta \tilde u + a - (b + 1) \tilde u + \tilde u^2 , \\
			\partial_t \tilde v &= d_2 \Delta \tilde v + b \tilde u - \tilde u^2 \tilde v ,  
		\end{split}
	\end{equation}
	for which $b$ is a control parameter, has a unique spatially homogeneous steady state $(\tilde u^\ast, \tilde v^\ast) = (a, b/a)$. This well-known reaction-diffusion system was proposed by the Brussels school as a model for certain kinds of chemical reactions \cite{Glansdorff1971}. Direct calculations in \cite[App.~B]{Kuramoto1984} show that the eigenvalue problem obtained after translating $(\tilde u^\ast, \tilde v^\ast)$ to the origin via $(u, v) = (\tilde u - a, \tilde v - b/a)$ and linearising about $\bm u^\ast = (0,0)$ has eigenfunctions $\widehat{\bm v} = (a_\xi, b_\xi)^\transpose \me^{i \bm \xi \cdot \bm x}$, where $\bm \xi = (\xi_1, \ldots, \xi_n) \in \R^n$, with eigenvalues $\lambda_{1,2}$ which satisfy
	\[
	\lambda_{1,2}^2 + \sigma(\xi,b) \lambda_{1,2} + \kappa(\xi,b) = 0,
	\]
	where $\xi = |\bm \xi|$ and
	\[
	\sigma(\xi,b) = 1 + a^2 - b + (d_1 + d_2) \xi^2 , \quad
	\kappa(\xi,b) = a^2 + (a^2 d_1 + (1 - b) d_2) \xi^2 + d_1 d_2 \xi^4 .
	\]
	An instability of the steady state $\bm u^\ast$ can occur if either (i) $\sigma(\xi,b) = 0$ or (ii) $\kappa(\xi,b) = 0$ for some $\xi = \xi_c$ and $b = b_c$ under variation in $b$ (recall that $b$ is a control parameter when $\eps = 0$). We are interested in the former case, which can be shown to occur at $\xi_c=0$ for $b_c = 1 + a^2$. The latter case is ruled out if
	\[
	\sqrt{\frac{d_1}{d_2}} > \frac{1}{a} \left( \sqrt{1 + a^2} - 1 \right) .
	\]
	We note that $d_1 > d_2$ would be sufficient, since the right-hand side is less than $1$. The graph of $\xi \mapsto \lambda_1(\xi,b)$ defines an $n$-dimensional surface, which for $b = b_c$ looks like a rotationally symmetric version of Fig.~\ref{fig:spectra}(d).
	
	In order to study the dynamics near the instability, we make one last coordinate transformation $\mu = (b - b_c) / b_c$. Together with the coordinate translation described above, this leads to
	\begin{equation}
		\label{eq:dynamic_b_2}
		\begin{split}
			\partial_t u &= d_1 \Delta u + a^2 u + a^2 v + (1 + a^2) u \mu + f(u,v,\mu) , \\
			\partial_t v &= d_2 \Delta v - (1 + a^2) u - a^2 v - (1 + a^2) u \mu - f(u,v,\mu) - \eps (1 + a^2) / a ,  \\
			\dot \mu &= \eps ,
		\end{split}
	\end{equation}
	where $f(u,v,\mu) := ( (1 + a^2) (1 + \mu) / a) u^2 + 2 a u v + u^2 v$. System \eqref{eq:dynamic_b_2} is in the general form \eqref{eq:dynamic_general_form} with
	\begin{equation}
		\label{eq:LM_b}
		\M = 
		\begin{pmatrix}
			a^2 & a^2 \\
			-(1 + a^2) & -a^2
		\end{pmatrix} ,
		\qquad
		\Lop = \textup{diag}\ (d_1, d_2) \Delta , 
	\end{equation}
	and
	\[
	\N = 
	\begin{pmatrix}
		(1 + a^2) u \mu + f(u,v,\mu) \\
		- (1 + a^2) u \mu - f(u,v,\mu) - \eps (1 + a^2) / a
	\end{pmatrix} .
	\]
	The preceding observations imply that the static subsystem
	\begin{equation}
		\label{eq:static_b_2}
		\begin{split}
			\partial_t u &= d_1 \Delta u + a^2 u + a^2 v + (1 + a^2) u \mu + f(u,v,\mu) , \\
			\partial_t v &= d_2 \Delta v - (1 + a^2) u - a^2 v - (1 + a^2) u \mu - f(u,v,\mu) ,
		\end{split}
	\end{equation}
	satisfies Assumptions \ref{ass:steady_state}-\ref{ass:bifurcation} (the eigenvalues $\lambda_{1,2}$ are invariant under the coordinate translations) with $\xi_\pm(\delta) = \pm ((1 + a^2) / (d_1 + d_2)) \delta + O(\delta^2)$, $\xi_c = 0$ and $\omega_c = a \neq 0$. 
	It follows that $\bm u^\ast = (0,0)$ undergoes a Hopf bifurcation at $\mu = 0$ in system \eqref{eq:static_b_2}, and a dynamic Hopf bifurcation in system \eqref{eq:dynamic_b_2}.
	
	\
	
	A formal modulation reduction for a class of reaction-diffusion systems containing the static subsystem \eqref{eq:static_b_2} can be found in \cite{Kuramoto1984} (the static Brusselator system \eqref{eq:static_b_1} in particular is treated in \cite[App.~B]{Kuramoto1984}). 
	The calculations therein lead to an approximation for $0 < \mu = \delta^2 \ll 1$ of the form
	\begin{equation}
		\label{eq:approximation_b}
		\begin{pmatrix}
			u(\bm x,t) \\
			v(\bm x,t)
		\end{pmatrix}
		= \delta \left( A(\bm X, T) \me^{iat} \bm \varphi
		+ c.c.\right) + O(\delta^2) , 
	\end{equation}
	where $\bm X = \delta \bm x$, $T = \delta^2 t$, $\bm \varphi = (1 , -1 + ia^{-1})^\transpose$ is the neutral eigenvector at criticality, and $A(\bm X, T) \in \mathbb C$ obeys the complex GL equation
	\begin{equation}
		\label{eq:modulation_eqn_static_b}
		\partial_T A = c_1 \Delta_{X} A + c_2 A - c_3 A |A|^2 
	\end{equation}
	where $\Delta_X := \sum_{j=1}^n \partial_{X_j}^2$, with coefficients
	\begin{equation}
		\label{eq:nu_i}
				c_1 = \frac{d_1 + d_2 - ia (d_1 - d_2)}{2}, \
				c_2 = \frac{1 + a^2}{2} , \ 
				c_3 = \frac{1}{2} \left( \frac{2 + a^2}{a^2} + i \frac{4 - 7a^2 + 	4a^4}{3a^3} \right) .
	\end{equation}
	Approximation and attractivity results for the modulation equation \eqref{eq:modulation_eqn_static_b} have been proven in \cite{Schneider1998}.
	
	\begin{remark}
		For $p = 1$, the complex GL equation \eqref{eq:modulation_eqn_static_b} is a universal modulation equation which governs the dynamics of temporally oscillating systems in the weakly nonlinear regime associated to a Hopf type instability. The static complex GL equation has been studied extensively; we refer to \cite{Aranson2002,Cross1993,Hoyle2006,Kuramoto1984,Schneider2017} and the references therein. More recently, the authors in \cite{Goh2022} studied the slow passage through a Hopf bifurcation in the complex GL equation itself by considering the case in which $\real c_2$ increases slowly in time.
	\end{remark}

	\subsection{(M3): Dynamic Turing-Hopf bifurcation}
	\label{sub:case_ii_dynamic_Turing-Hopf_bifurcation}
	
	We consider a dynamic Turing-Hopf bifurcation in a system of coupled Kuramoto–Sivashinsky equations
	\begin{equation}
		\label{eq:dynamic_ks}
		\begin{split}
			\partial_t u &= -(1 + \partial_x^2)^2 u - \partial_x u + \mu u + \partial_x (u^2 + uv + v^2) , \\
			\partial_t v &= -(1 + \partial_x^2)^2 v + \partial_x v + \mu v + \partial_x (u^2 + uv + v^2) , \\
			\dot \mu &= \eps ,
		\end{split}
	\end{equation}
	where $x \in \Omega = \R$, $\bm u(x,t) = (u(x,t), v(x,t)) \in \R^2$, $\mu(t) \in \R$ and $0 < \eps \ll 1$. The static subsystem \eqref{eq:dynamic_ks}$|_{\eps = 0}$ appears in \cite{Schneider2017} as a simplified version of the system considered in \cite{Schneider1997}. Both can be viewed as extensions of the classical Kuramoto-Sivashinsky equation $\partial_t u = -(1+\partial_x^2)^2 u + \mu u - \partial_x(u^2)$, which was originally derived to model flame propagation \cite{Kuramoto1980,Kuramoto1976,Sivashinsky1977}.
	
	System \eqref{eq:dynamic_ks} is already in the form \eqref{eq:dynamic_general_form}, with
	\begin{equation}
		\label{eq:LM_ks}
		\M = \mathbb O_{2,2}, \qquad 
		\Lop = - (1 + \partial_x^2)^2 \mathbb{I}_2 + \partial_x
		\begin{pmatrix}
			- 1 & 0 \\
			0 & 1
		\end{pmatrix} , 
	\end{equation}
	where $\mathbb O_{2,2}$ and $\mathbb{I}_2$ denote $2 \times 2$ zero and identity matrices respectively, and
	\[
	\N(\bm u,\mu,\eps) = 
	\left( \mu u + \partial_x (u^2 + uv + v^2) \right)
	\begin{pmatrix}
		1 \\
		1
	\end{pmatrix}.
	\]
	The static subsystem is given by 
	\begin{equation}
		\label{eq:static_ks}
		\begin{split}
			\partial_t u &= -(1 + \partial_x^2)^2 u - \partial_x u + \mu u + \partial_x (u^2 + uv + v^2) , \\
			\partial_t v &= -(1 + \partial_x^2)^2 v + \partial_x v + \mu v + \partial_x (u^2 + uv + v^2) .
		\end{split}
	\end{equation}
	Direct calculations show that Assumption \ref{ass:steady_state} is satisfied, and that linearisation about $\bm u^\ast = (0,0)$ leads to eigenfunctions $\widehat{\bm v} = (a_\xi , b_\xi)^\transpose \me^{i \xi x}$ with $a_\xi, b_\xi \in \mathbb C$, $\xi \in \R$ and two eigenvalues curves $\xi \mapsto \lambda_{1,2}(\xi,\mu)$ defined by
	\[
	\lambda_{1,2}(\xi,\mu) = -(1-\xi^2)^2 \mp i \xi + \mu, \qquad \xi \in \R.
	\]
	This expression can be used to verify Assumptions \ref{ass:symmetry}-\ref{ass:bifurcation}. Assumption \ref{ass:symmetry} is easily verified. For Assumption \ref{ass:bifurcation}, note that $\real \lambda_1$ coincides with the eigenvalue \eqref{eq:spectrum_sh} for the Swift-Hohenberg problem. Thus Assumption \ref{ass:bifurcation} is satisfied with $\xi_\pm(\delta)$ given by \eqref{eq:xi_pm_sh}. Since $\xi_c = 1 \neq 0$ and $\omega_c = - 1$, 
	we conclude that the static subsystem \eqref{eq:static_ks} has a Turing-Hopf bifurcation with critical modes $\xi = \pm \xi_c = \pm 1$ at $\mu = 0$, and that system \eqref{eq:dynamic_ks} has a dynamic Turing-Hopf bifurcation.
	
	\
	
	A modulation reduction 
	for the static subsystem \eqref{eq:static_ks} can be found in \cite{Schneider1997}. If we set $0 < \mu = \delta^2 \ll 1$, then the relevant ansatz is
	\begin{equation}
		\label{eq:approximation_ks}
		\begin{split}
			u &= \delta A_L(X_L, T) \me^{i (x - t)} + \delta^2 A_2 (X_L, T) \me^{2 i (x - t)} 
			+ \frac{1}{2} \delta^2 A_0^0 (X_L, T) + c.c., \\
			v &= \delta B_R(X_R, T) \me^{i (x + t)} + \delta^2 B_2 (X_R, T) \me^{2 i (x + t)} 
			+ \frac{1}{2} \delta^2 B_0^0 (X_R, T) + c.c.,
		\end{split}
	\end{equation}
	where $X_L = \delta (x + t)$, $X_R = \delta (x - t)$, $T = \delta^2 t$, and the $L,R$ notation emphasises the fact that $A_L$, $B_R$ denote amplitudes of left, right moving wave trains respectively. The following system of coupled complex GL-type equations can be derived:
	\begin{equation}
		\label{eq:modulation_eqn_static_ks}
		\begin{split}
			\partial_T A_L &= 4 \partial_{X_L} A_L + A_L - \tilde \gamma_1 A_L |A_L|^2 - \tilde \gamma_2 A_L |B_R|^2 , \\
			\partial_T B_R &= 4 \partial_{X_R} B_R + B_R - \tilde \gamma_1 B_R |B_R|^2 - \tilde \gamma_2 B_R |A_L|^2 ,
		\end{split}
	\end{equation}
	with constant coefficients $\tilde \gamma_1, \tilde \gamma_2 \in \mathbb C$ (see \cite{Schneider1997} or \cite[Ch.~10.8]{Schneider2017} for exact values). 
	A classical complex GL equation can only be obtained if $A_L=0$ or $B_R=0$ (both subspaces are invariant). An approximation theorem for system \eqref{eq:modulation_eqn_static_ks} is given in \cite{Schneider1997}. 
		
		
		\begin{remark}
			\label{rem:averaging}
			The derivation of system \eqref{eq:modulation_eqn_static_ks} involves the use of normal form transformations and averaging theory in order to push highly oscillatory in time terms back to higher orders in $\delta$.
		\end{remark}
		
		\begin{remark}
			For $p = 1$, coupled systems of generalised complex GL equations arise generically after modulation reduction in the weakly nonlinear regime associated to a Turing-Hopf bifurcation \cite{Cross1993,Frohoff2023}. One important example is the case of Taylor-Couette flow induced by the Navier-Stokes equations between strongly rotating cylinders; see \cite{Schneider1999} for a derivation of the (coupled complex GL type) modulation equations and corresponding approximation results.
		\end{remark}
		
		\subsection{(M4): Conserved dynamic long-wave bifurcation}
		\label{sub:case_iii_dynamic_stationary_long_wave_bifurcation}
		
		We consider a dynamic extension of a planar fluid flow problem appearing in \cite{Meshalkin1961}, where the statement of the problem is attributed to Kolmogorov. Subsequent treatments appeared in \cite{Nepomniashchii1976,Schneider1999b}. 
		
		The problem arises under the assumption of external forcing parallel to the $x$-axis, given by $\bm F(t,y) = A(t) \cos y \bm e_x$, where $\bm e_x = (1,0)^\transpose$. The primary difference in our formulation is that we allow the forcing amplitude to vary slowly in time according to $\dot A(t) = \tilde \eps$, where $0 < \tilde \eps = (\rho \nu^2 / \pi^3) \eps \ll 1$ and $\rho$, $\nu$ denote the constant density, viscosity of the fluid respectively. We assume that $\bm x = (x,y) \in \Omega = \R \times \widetilde \Omega = \R \times (-\pi, \pi)$, and require that the velocity field $\bm U = (u,v)$ and pressure $p$ satisfy the Navier-Stokes equations for an incompressible flow. This leads to
		\begin{equation}
			\label{eq:dynamic_ns_original}
				\partial_t \bm U = - \nabla p - (\bm U \cdot \nabla) \bm U + \Delta \bm U + \mathcal R \cos y \bm e_x , \qquad
				\dot {\mathcal R} = \eps , 
		\end{equation}
		where $\mathcal R(t) = A(t) \pi^3 / \rho \nu^2$ denotes the Reynolds number, together with the incompressibility and zero mean flow conditions
		\[
		\nabla \cdot \bm U = 0 , \qquad
		[u]_{\widetilde \Omega}(x) = \frac{1}{| \widetilde \Omega |} \int_{\widetilde \Omega} u(x,y) dy = 0 ,
		\]
		respectively (the latter is imposed as an additional constraint). We assume periodic boundary conditions for both $\bm U$ and $p$ in the $y$-direction.
		
		
		The static subsystem obtained by setting $\eps = 0$ in system \eqref{eq:dynamic_ns_original} is
		\begin{equation}
			\label{eq:static_ns_original}
			\partial_t \bm U = - \nabla p - (\bm U \cdot \nabla) \bm U + \Delta \bm U + \mathcal R \cos y \bm e_x ,
		\end{equation}
		which has a trivial steady state
		\[
		\bm U^\ast = \left( u^\ast, v^\ast \right) = (\mathcal R \sin y, 0) , \qquad p^\ast = const.
		\]
		The calculations in \cite{Meshalkin1961,Nepomniashchii1976} show that an instability appears at the critical Reynolds number $\mathcal R = \mathcal R^\ast = \sqrt 2$ (as can also be seen from the expression for the leading eigenvalue $\lambda(\xi,\mathcal R')$ below). Setting $\bm U' = (u',v') := (u - u^\ast,v - v^\ast)$, $p' = p - p^\ast$ and $\mathcal R' = \mathcal R - \mathcal R^\ast$ leads to the following system with an instability of the trivial solution $(u',v') = (0,0)$ at $p' = \mathcal R' = 0$:
		\begin{equation}
			\label{eq:dynamic_ns}
				\partial_t \bm U' = - (\bm U' \cdot \nabla) \bm U' - \nabla p' + \Delta \bm U' - (\mathcal R' + \mathcal R^\ast) \textup{W} \bm U' + \eps \bm Q , \qquad
				\dot{\mathcal R'} = \eps ,
		\end{equation}
		where
		\[
		\textup{W} = 
		\begin{pmatrix}
			\sin y \ \partial_x & \cos y \\
			0 & \sin y \ \partial_x
		\end{pmatrix}
		, \qquad 
		\bm Q =
		\begin{pmatrix}
			- \sin y \\
			0
		\end{pmatrix} ,
		\]
		together with
		\begin{equation}
			\label{eq:ns_constraints}
			\nabla \cdot \bm U' = 0 , \qquad
			[u']_{\widetilde \Omega} = \frac{1}{| \widetilde \Omega |} \int_{\widetilde \Omega} u'(x,y) dy = 0 .
		\end{equation}
		In order to write this in the general form \eqref{eq:dynamic_general_form}, we append the trivial equation $\partial_t p = 0$. This allows us to write the equations in the form \eqref{eq:dynamic_general_form} with $\bm u = (\bm U', p)^\transpose$,
		\begin{equation}
			\label{eq:LM_ns}
			\M = 
			\begin{pmatrix}
				0 & - \mathcal R^\ast \cos y & 0 \\
				\ & \mathbb O_{2,3} & \ 
			\end{pmatrix} , 
			\qquad 
			\Lop = 
			\begin{pmatrix}
				- \mathcal R^\ast \sin y \mathbb I_2 \partial_x + \Delta & \nabla^\transpose \\
				\mathbb O_{1,2} & 0
			\end{pmatrix} ,
		\end{equation}
		where $\nabla^\transpose = (\partial_x, \partial_y)^\transpose$ and $\mathbb O_{j,k}$ denotes the $j \times k$ zero matrix, and
		\begin{equation}
			\label{eq:N_ns}
			\N(\bm U',p',\mathcal R',\eps) = 
			\begin{pmatrix}
				- (\bm U' \cdot \nabla) \bm U' 
				- \mathcal R' \textup{W} \bm U' + \eps \bm Q \\
				0
			\end{pmatrix} .
		\end{equation}
		The static subsystem for $\eps = 0$, i.e.
		\begin{equation}
			\label{eq:static_ns}
			\partial_t \bm U' = - (\bm U' \cdot \nabla) \bm U' - \nabla p' + \Delta \bm U' - (\mathcal R' + \mathcal R^\ast) \textup{W} \bm U' , \qquad 
			\partial_t p = 0,
		\end{equation}
		satisfies Assumption \ref{ass:steady_state} with $\bm u^\ast = ({\bm U'}^\ast, p^\ast) = (\bm 0, 0)$. In particular, the leading eigenvalue curve has the following expansion about $\xi=0$:
		\[
		\lambda(\xi,\mathcal R') = - \left(1 - \frac{(\mathcal R' + \mathcal R^\ast)^2}{2}\right) \xi^2 - (\mathcal R' + \mathcal R^\ast)^2 \left( 1 + \frac{(\mathcal R' + \mathcal R^\ast)^2}{4} \right) \xi^4 + O(\xi^6) ,
		\]
		see e.g.~\cite{Nepomniashchii1976,Schneider1999b}. Note that $\xi_c = 0$. The ordering requirement $\real \lambda_{j+1} \geq \real \lambda_j$ in Assumption \ref{ass:symmetry} follows from the fact that we impose periodic boundary conditions on $\widetilde \Omega$ together with the translation invariance of the linear operator $\M + \Lop$ in $x$.
		
		\
		
		The present problem is, however, distinguished from the others considered in this work due to the fact that Assumption \ref{ass:bifurcation} (i) and (iii) are violated. This is because $\lambda(0,\mathcal R') = 0$ for all $\mathcal R' \in \R$. This situation is typical in the presence of a \textit{conservation law} \cite{Cross1993,Frohoff2023}. It arises in the present case because the incompressibility condition is equivalent to the conservation of mass, since
		\[
		\rho = const. \ \implies \ \frac{\partial \rho}{\partial t} + \rho \nabla \cdot \bm U' = \rho \nabla \cdot \bm U' = 0 .
		\]
		The static subsystem \eqref{eq:static_ns} satisfies a slightly weaker version of Assumption \ref{ass:bifurcation} for which the conditions (i) and (iii) are replaced by
		\begin{enumerate}
			\item[(i)$'$] $\mu < 0 \ \implies \ \real \lambda_1(\xi,\mu) < 0$ for 
			all $\xi \neq 0$;
			\item[(iii)$'$] $\mu > 0 \ \implies \ \real \lambda_1(\xi,\mu) > 0$ for all $\xi$ such that $|\xi| \in (\xi_-(\delta), 0) \cup (0, \xi_+(\delta))$, where $\xi_\pm(\delta)$ are continuous functions depending on $\mu =: \delta^2 \ll 1$ satisfying $\xi_-(0) = \xi_+(0)$. Moreover, $\real \lambda_1(\xi_\pm(\delta),\mu) = 0$ and $\real \lambda_1(\xi,\mu) < 0$ for all $\xi$ such that $|\xi| \notin [\xi_-(\delta), \xi_+(\delta)]$.
		\end{enumerate} 
		Direct calculations show that the modified Assumption \ref{ass:bifurcation}$'$ is satisfied with $\xi_\pm(\delta) = \pm \sqrt{2 / 3} \delta + O(\delta^2)$. Since $\xi_c = \omega_c = 0$ and the static subsystem \eqref{eq:static_ns} satisfies Assumptions \ref{ass:steady_state}, \ref{ass:symmetry} and \ref{ass:bifurcation}$'$, we say that there is a \textit{conserved long-wave bifurcation} at $\mathcal R' = 0$ and, therefore, that system \eqref{eq:dynamic_ns} undergoes a \textit{conserved dynamic long-wave bifurcation}.
		
		\
		
		
		A modulation reduction for the static subsystem \eqref{eq:static_ns} can be found in \cite{Nepomniashchii1976}. The idea is to expand in powers of $\mathcal R' = \delta^2 \ll 1$ as follows:
		\begin{equation}
			\label{eq:approximation_ns}
			\begin{split}
				u'(x,y,t) &= \delta \sum_{j=0}^\infty u^{(j)}(X,y,T) \delta^j , \\
				v'(x,y,t) &= \delta \sum_{j=0}^\infty v^{(j)}(X,y,T) \delta^j , \\
				p'(x,y) &= \delta^2 \sum_{j=0}^\infty p^{(j)}(X,y) \delta^j , 
			\end{split}
		\end{equation}
		where $X = \delta x$ and $T = \delta^4 t$. The formal calculations leads to an approximation of the form
		\[
		\begin{pmatrix}
			u'(x,y,t) \\
			v'(x,y,t)
		\end{pmatrix}
		= 
		\delta A(X,T) \bm \varphi
		+ O(\delta^2) ,
		\]
		where $\bm \varphi = (-\sqrt 2 \cos y, 1)^\transpose$ and the function $A(X,T) = v^{(0)}(X,T) \in \R$ satisfies a modulation equation of Cahn-Hilliard type:
		\begin{equation}
			\label{eq:modulation_eqn_static_ch}
			\partial_T A = - 3 \partial_X^4 A - 3 \partial_X^2 A + \frac{2}{3} \partial_X^2 \left( A^3 \right) .
		\end{equation}
		Approximation and attractivity results for the modulation equation \eqref{eq:modulation_eqn_static_ch} have been proven in \cite{Schneider1999b}.

		\section{Geometric blow-up and the method of multiple scales}
		\label{sec:geometric_blow-up_and_the_modulation_equations}
		
		In this section we apply the geometric blow-up method developed for PDEs in \cite{Jelbart2022} to the general system \eqref{eq:dynamic_general_form}. We need to extend and generalise the method, which has so far only been developed for the study of dynamic Turing bifurcations, so that it can also be used to study dynamic bifurcations of Hopf, Turing-Hopf and long-wave type. This is achieved in two steps:
		\begin{enumerate}
			\item[(I)] We apply a weighted geometric blow-up transformation centered at $\bm u = \bm u^\ast = \bm 0$ and $\mu = \eps = 0$, without imposing any specific structure on $\bm u$ other than $\bm u = O(r)$, where $r \geq 0$ measures the distance from the blow-up manifold.
			\item[(II)] We apply a dynamic variant of the method of multiple scales to the blown-up problem, in order to formally reduce the equations to simpler equations of modulation type.
		\end{enumerate}
		As noted above, Step (I) has already been achieved for the Swift-Hohenberg problem in \cite{Jelbart2022}. The primary contribution of the present article is to show that Step (II) provides a natural and systematic setting for extending this work so that it can be applied to study a larger class of bifurcations, and therefore to the emergence of spatial, temporal and spatio-temporal patterns in the fast-slow setting. It is worthy to emphasise at this stage that the method of multiple scales is a purely \textit{formal} asymptotic method (see e.g.~\cite{Kevorkian2012,Nayfeh2008}). We therefore only obtain formal results by the approach outlined by the Steps (I)-(II) above. In the absence of a spectral gap and center manifold type theorems, however, this is to be expected. For a rigorous treatment, a third step is necessary:
		\begin{enumerate}
			\item[(III)] Show that solutions of the original system \eqref{eq:general_pde} and the modulation equation obtained via Steps (I)-(II) are `close'.
		\end{enumerate}
		Since the modulation equation obtained via Steps (I)-(II) is posed in the blown-up space, the error estimates must also be obtained in the blown-up space. We are primarily focused on Steps (I)-(II) in this work. Step (III), which has already been carried out for the Swift-Hohenberg equation in \cite{Jelbart2022}, is left for future work.
		
		\begin{remark}
			Steps (I)-(III) are in direct correspondence with Steps (I)-(III) from Section \ref{sec:introduction}.
		\end{remark}
		
		\begin{remark}
			The aim in Step (III) is to prove the so-called `approximation property' in the fast-slow setting. The other important property for a modulation equation -- attractivity -- is expected to hold for large spaces of initial conditions in the strongly attracting regime $\mu(0) < -c < 0$. Since this is the typical situation in fast-slow applications, we omit the additional requirement in Step (III).
		\end{remark}

		\subsection{Step (I): Geometric blow-up}
		\label{sub:geometric_blow_up}
		
		Consider the general system \eqref{eq:dynamic_general_form}. In order to apply the geometric blow-up method, we first extend the system by including $\eps$ as an additional variable, i.e.~we consider
		\begin{equation}
			\label{eq:dynamic_general_extended}
			\partial_t \bm u = \left( \M + \Lop \right) \bm u + \N(\bm u, \mu, \eps) , \qquad
			\dot \mu = \eps , \qquad
			\dot \eps = 0 .
		\end{equation}
		We consider only cases for which the static equation with $\eps = 0$, i.e.~equation \eqref{eq:static_general_form}, has a bifurcation of the stationary state $\bm u = \bm 0$ at $\mu = 0$ which is known to be well-approximated for $0 < \mu = \delta^2 \ll 1$ by a static multi-scale ansatz of the form
		\begin{equation}
			\label{eq:static_ansatz}
			\bm u(\bm x,t) = \delta \bm \psi(\ \cdots, \bm X, T) ,
		\end{equation}
		where $\bm X = (\delta \bm x_{ub}, \bm x_b)$, $T = \delta^\beta t$ for some positive integer $\beta$, and the `$\cdots$' notation indicates that $\bm \psi$ may or may not depend on the original space $\bm x$ and time $t$. Based on the form of the multi-scale ansatz for the model problems presented in Section \ref{sec:dynamic_Turing_instability}, we are interested in all four possibilities:
		\begin{equation}
			\label{eq:psi_options}
			\bm \psi(x, \bm X, T) , \qquad
			\bm \psi(t, \bm X, T) , \qquad
			\bm \psi(x, t, \bm X, T) , \qquad
			\bm \psi(\bm X, T) ,
		\end{equation}
		recall equations \eqref{eq:approximation_sh}, \eqref{eq:approximation_b}, \eqref{eq:approximation_ks} and \eqref{eq:approximation_ns} respectively.
		
		The static transformation \eqref{eq:static_ansatz} can be `made dynamic' by reformulating it in terms of a geometric blow-up transformation
		\[
		\Phi : [0,\infty) \times \mathcal X \times S^1 \to \mathcal X \times \R \times [0,\eps_0] 
		\]
		defined by
		\begin{equation}
			\label{eq:general_blow-up}
			r(\bar t) \geq 0, \ \bm \psi(\cdots, \cdot, \bar t) \in \mathcal X, \ (\bar \mu, \bar \eps) \in S^1 \mapsto
			\begin{cases}
				\bm u(\cdot,t) = r(\bar t) \bm \psi(\cdots, \cdot,\bar t) , \\
				\mu = r(\bar t)^2 \bar \mu(\bar t) , \\
				\eps = r(\bar t)^{2 + \beta} \bar \eps(\bar t) ,
			\end{cases}
		\end{equation}
		where $\mathcal X$ is a suitable Banach space (we need not specify it here as we focus purely on the formal asymptotic analysis), $S^1$ is the unit circle in $\R^2$, and 
		the so-called \textit{desingularized} space $\bar{\bm x}$ and time $\bar t$ are defined via
		\begin{equation}
			\label{eq:general_desingularization}
			\partial_{x_j} = 
			\begin{cases}
				r(\bar t) \partial_{\bar x_j} , & j = 1, \ldots, p, \\
				\partial_{x_j} , & j = p+1, \ldots, n , 
			\end{cases}
			\qquad 
			\partial_t = r(\bar t)^\beta \partial_{\bar t} ,
		\end{equation}
		where $\beta$ is the same positive integer as in \eqref{eq:static_ansatz}. Note that only the $x_j$ corresponding to unbounded spatial directions are desingularized. Geometrically, the set $\{\mu = \eps = 0\}$ is blown up to the `cylinder' $\{r = 0\} \times \mathcal X \times S^1$. This is sketched in Figure \ref{fig:blow_up}; see also the caption for additional explanations.
		
		\begin{figure}[t!]
			\centering
			\includegraphics[width=0.99\textwidth]{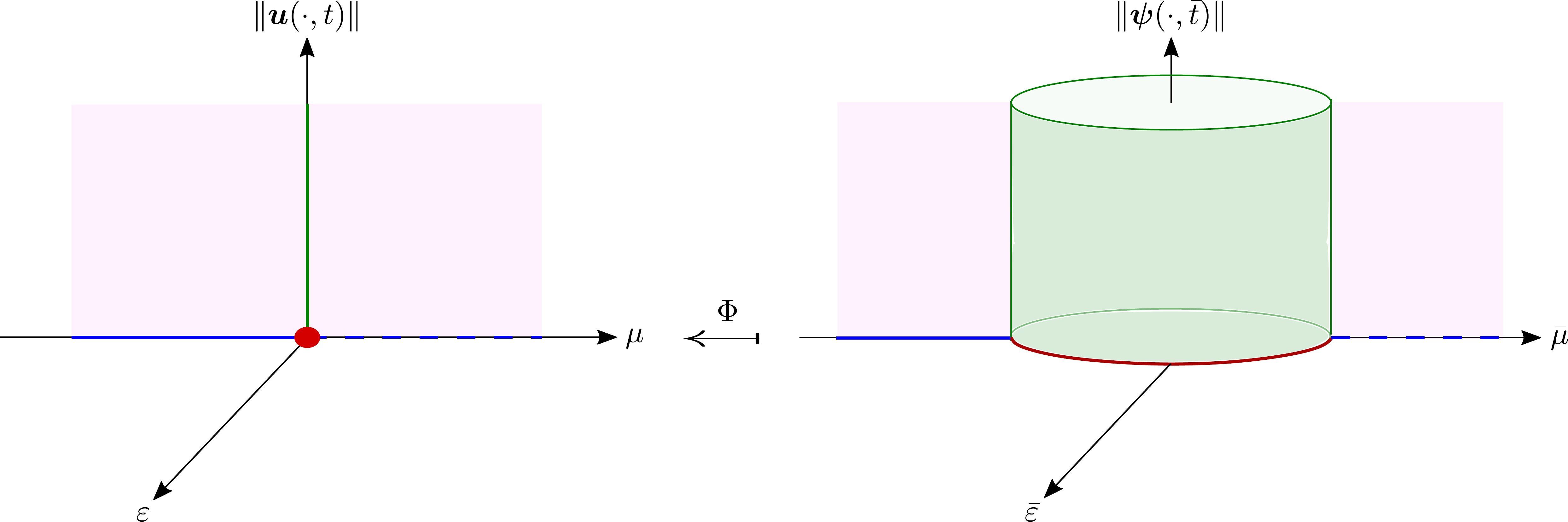}
			\caption{The geometry before (left) and after (right) applying the blow-up map $\Phi$ defined by \eqref{eq:general_blow-up}-\eqref{eq:general_desingularization}. Here $\| \cdot \|$ denotes the norm in the (unspecified) Banach space $\mathcal X$. The set $\{\mu = \eps = 0\}$ (green) in the original space is blown-up to the `cylinder' (shaded green) in the blown-up space. Both $\{\mu < 0, \eps = 0 \}$ and $\{\mu > 0, \eps = 0 \}$, which are invariant under $\Phi$, and their preimages are shown in shaded magenta. The linearly stable and unstable branches $\{ (\bm u, \mu, \eps) = (\bm u^\ast,\mu,0) : \mu < 0 \}$ and $\{ (\bm u, \mu, \eps) = (\bm u^\ast,\mu,0) : \mu > 0 \}$ respectively, as well as their preimages under $\Phi$, are indicated by blue and dashed blue lines respectively. The singularity at the origin and its preimage under $\Phi$ are shown in red.}
			\label{fig:blow_up}
		\end{figure}
		
		The form of the blow-up map defined by \eqref{eq:general_blow-up}-\eqref{eq:general_desingularization} is determined by two factors:
		\begin{itemize}
			\item The form of the corresponding static ansatz \eqref{eq:static_ansatz};
			\item The shape of the spectrum associated to the static subsystem when $\mu = \delta^2 \ll 1$.
		\end{itemize}
		The form of the static ansatz for $\bm u$ determines the form of the dynamic ansatz, because one simply replaces the small parameter $\delta$ with the time-dependent function $r(\bar t) \geq 0$ which measures the radius of solutions from the static bifurcation point. 
		Since $\delta$ has been replaced by a time-dependent variable $r(\bar t)$, the simple rescalings $\bm X = (\delta \bm x_{ub}, \bm x_b)$ and $T = \delta^\beta t$ need to be replaced by the desingularizations in \eqref{eq:general_desingularization}. The `weights' associated to the desingularization (i.e.~the exponents of $r$), can be read directly off of the corresponding weights for the (presumably known) static ansatz, which are determined by the spectrum of the static problem for $0 < \mu = \delta^2 \ll 1$; recall Remark \ref{rem:scaling}. Of course, $\mu(t)$ and $\eps$ are also coupled to the blow-up ansatz via their dependence on $r(\bar t)$. The relevant weights for their defining equations are determined by the requirement that the blown-up vector field is well-defined and non-trivial as $r \to 0$.
		
		\begin{remark}
			The blow-up transformation in defined by \eqref{eq:general_blow-up}-\eqref{eq:general_desingularization} reduces to the static transformation \eqref{eq:static_ansatz} after an invariant restriction to $\{ (\bar \eps,\bar \mu) = (0,1) \in S^1,  \mu = r^2 = const. > 0 \}$. This shows that the blow-up transformation can be viewed as a `dynamic generalisation' of the classical parameter dependent ansatz.
		\end{remark}
		
		\begin{remark}
			Applications with spatial anisotropy are expected to require a more complicated spatial desingularization of the form
			\[
			\partial_{x_j} = r(\bar t)^{\nu_j} \partial_{\bar x_j} , \qquad 
			j = 1, \ldots, p,
			\]
			for (possibly distinct) $\nu_j > 0$. In order to study the emergence of anisotropic rolls emerging parallel to the $x$-axis associated with a dynamic Turing bifurcation in the Newell-Whitehead equation, for example, the relevant desingularization is expected to take the form
			\[
			\partial_{x_1} = r(\bar t) \partial_{\bar x_1} , \qquad 
			\partial_{x_2} = r(\bar t)^{1/2} \partial_{\bar x_2} ,
			\]
			since the static problem features an $O(\delta)$-wide band of unstable modes $\xi_1$ and an $O(\delta^{1/2})$-wide band of unstable modes $\xi_2$ \cite{Newell1969,Segel1969}. 
			We do not consider this more complicated case in this work, since the particular form in \eqref{eq:general_desingularization} for which $\nu_j = 1$ for all $j = 1, \ldots, p$ is sufficient to analyse the model problems we consider.
		\end{remark}
		
		Applying the blow-up map to system \eqref{eq:dynamic_general_extended} leads to the following equation, posed in global coordinates in the blown-up space:
		\begin{equation}
			\label{eq:psi_eqn}
			\partial_{\bar t} \bm \psi = r^{-\beta} \left(\M + \Lop \right) \bm \psi - r^{-1} \bm \psi \partial_{\bar 	t} r + r^{-1 - \beta} \N(r \bm \psi, r^2 \bar \mu, r^{2 + \beta} \bar \eps) .
		\end{equation}
		For the time being, we continue to view $\Lop$ and $\N$ as functions of $\bm x$, as opposed to $\bar{\bm x}$. As we shall see in the following, a suitable choice for $\bm \psi$ and $\beta$ (depending on the problem) guarantees that the remaining terms are well-defined and non-trivial as $r \to 0$.
		
		\begin{figure}[t!]
			\centering
			\includegraphics[width=0.45\textwidth]{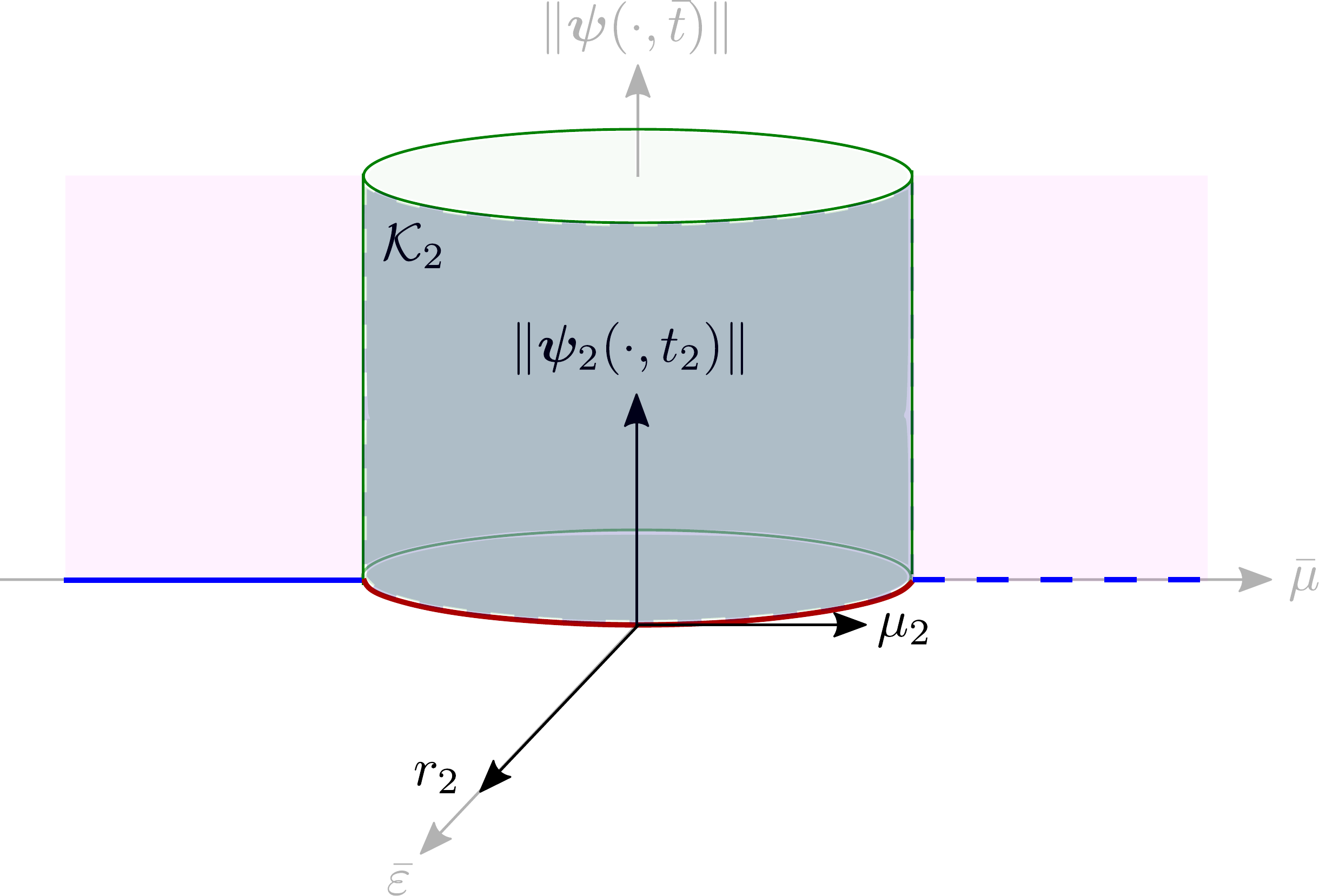}
			
			\includegraphics[width=0.45\textwidth]{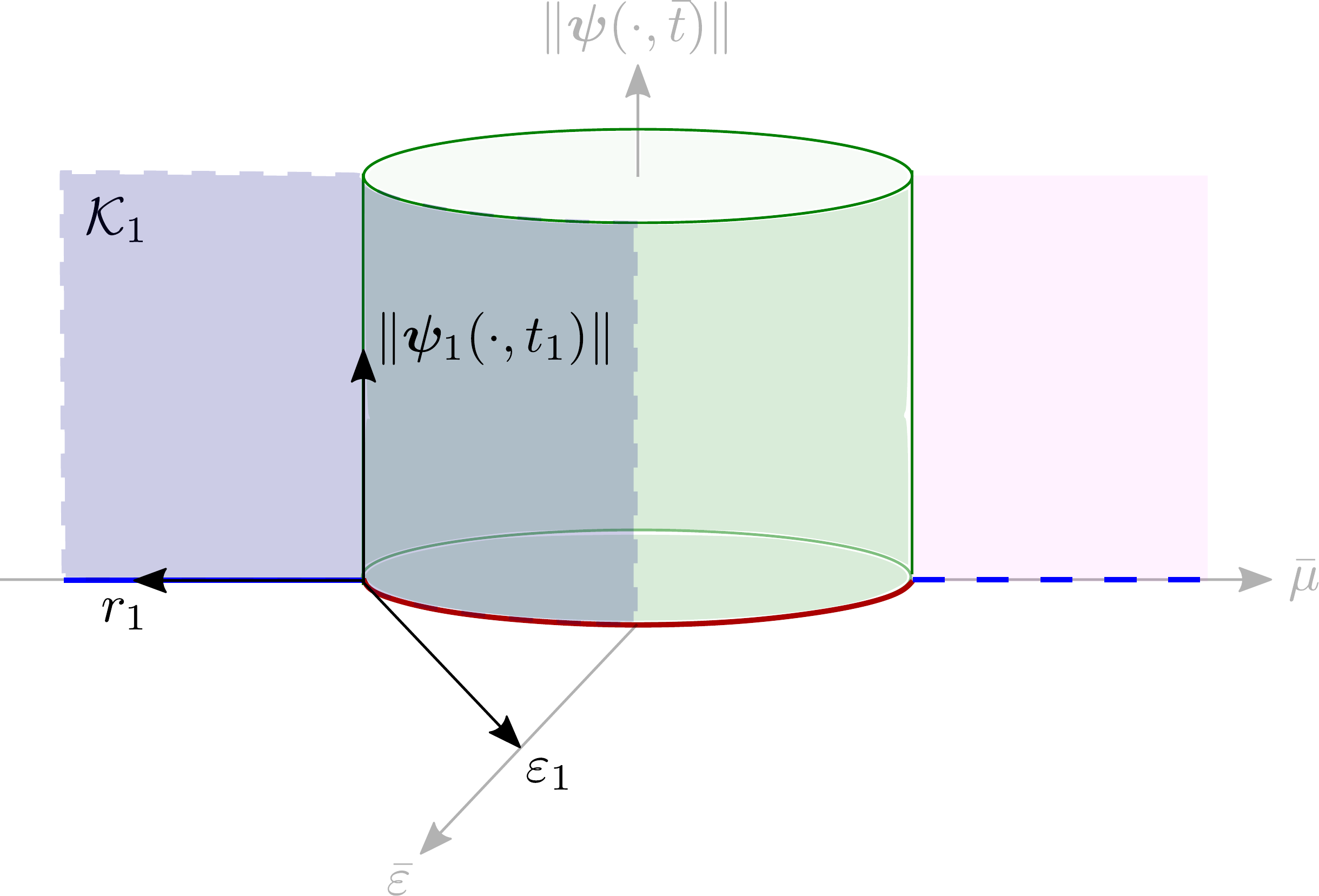} \qquad
			\includegraphics[width=0.45\textwidth]{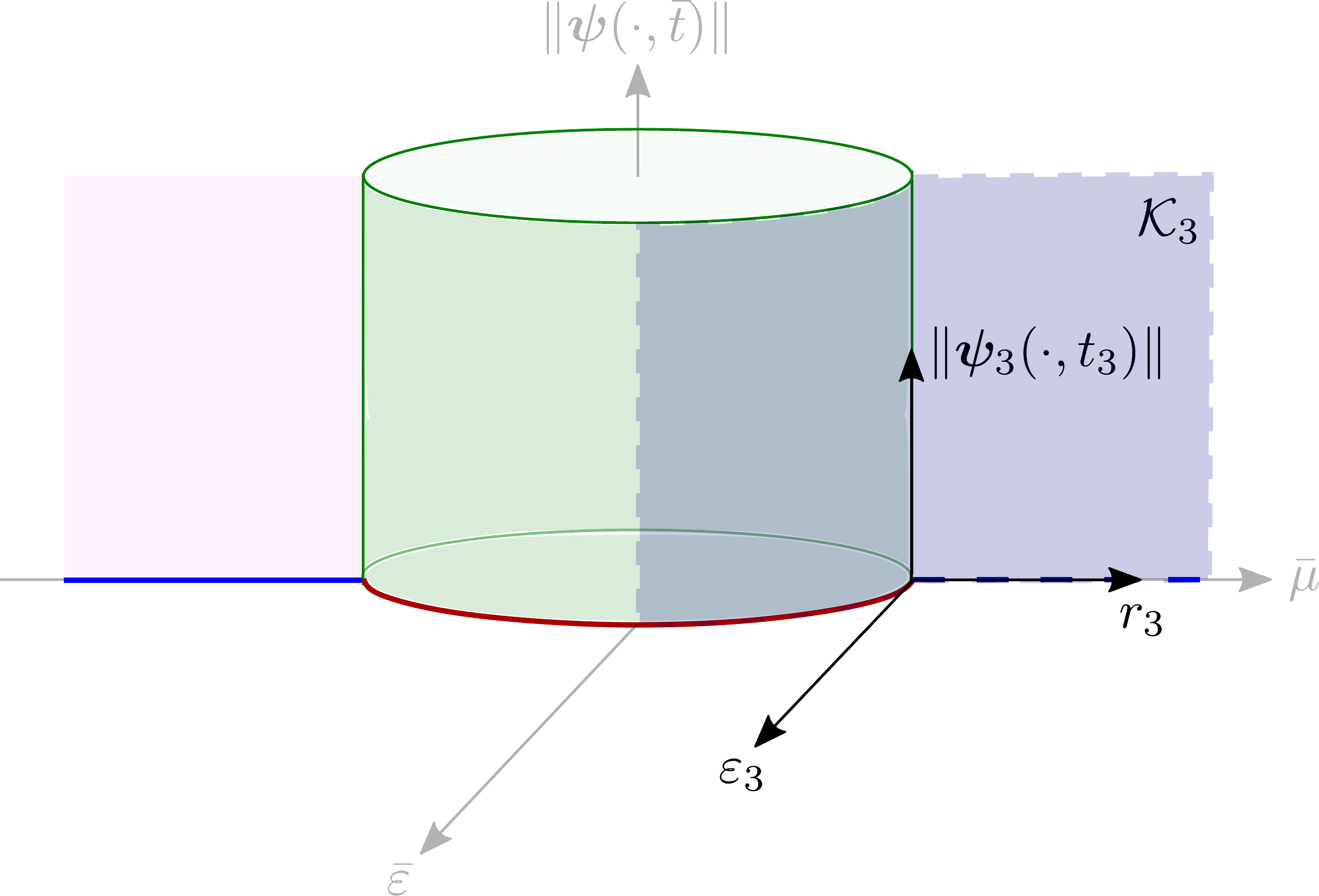}
			\caption{Local coordinate axes for the coordinate charts $\mathcal K_i$ defined by \eqref{eq:coordinates}, $i=1,2,3$. Projections of the `visible part' of blown-up space onto the invariant subspace defined by $\{r = 0 \} \cup \{ \eps = 0\}$ are shown in shaded blue. Solutions approach a neighbourhood of the blow-up cylinder in the region visible in $\mathcal K_1$ (bottom left), pass `over' the blow-up cylinder via the region visible in $\mathcal K_2$ (top), and leave a neighbourhood of the blow-up cylinder via the region visible in $\mathcal K_3$ (bottom right). The dynamics in overlapping charts are related by the smooth transformations in \eqref{eq:kappa_ij}.}
			\label{fig:charts}
		\end{figure}
		
		For specific calculations it is convenient to work in local coordinate charts which we denote by
		\[
		\mathcal K_1 : \ \bar \mu = -1, \qquad 
		\mathcal K_2 : \ \bar \eps = 1, \qquad 
		\mathcal K_3 : \ \bar \mu = 1,
		\]
		and for which we introduce the following chart-specific coordinates:
		\begin{equation}
			\label{eq:coordinates}
			\begin{aligned}
				\mathcal K_1 &: \ (\bm u, \mu, \eps) = (r_1 \bm \psi_1, -r_1^2, r_1^{2 + \beta} \eps_1) , \\
				\mathcal K_2 &: \ (\bm u, \mu, \eps) = (r_2 \bm \psi_2, r_2^2 \mu_2, r_2^{2 + \beta}) , \\
				\mathcal K_3 &: \ (\bm u, \mu, \eps) = (r_3 \bm \psi_3, r_3^2, r_3^{2 + \beta} \eps_3) .
			\end{aligned}
		\end{equation}
		The local coordinates in each chart are represented in Figure \ref{fig:charts}. We shall also denote the time and space in chart $\mathcal K_i$ (defined via \eqref{eq:general_desingularization}) by $t_i$ and $\bm x_i$, respectively. If necessary we will write $\bm x_i = (x_{1,i},\ldots,x_{n,i})$ to indicate the componentwise notation of $\bar x$ in charts, however we shall also continue to write $\bm x = (x_1, \ldots, x_n)$ (the meaning will be clear from context).
		
		In chart $\mathcal K_1$ we obtain the equations
		\begin{equation}
			\label{eq:K1_eqns}
			\begin{split}
				\partial_{t_1} \bm \psi_1 &= r_1^{-\beta} \left(\M + \Lop \right) \bm \psi_1 + \frac{1}{2} \bm \psi_1 \eps_1 + r_1^{-1-\beta} \N(r_1 \bm \psi_1, -r_1^2, r_1^{2+\beta}\eps_1 ) , \\
				\dot r_1 &= - \frac{1}{2} r_1 \eps_1 , \\
				\dot \eps_1 &= \frac{2 + \beta}{2} \eps_1^2 ,
			\end{split}
		\end{equation}
		in chart $\mathcal K_2$ we obtain the equations
		\begin{equation}
			\label{eq:K2_eqns}
			\begin{split}
				\partial_{t_2} \bm \psi_2 &= r_2^{-\beta} \left(\M + \Lop \right) \bm \psi_2 + r_2^{-1-\beta} \N(r_2 \bm \psi_2, r_2^2 \mu_2, r_2^{2+\beta} ) , \\
				\dot \mu_2 &= 1 ,
			\end{split}
		\end{equation}
		where $\dot r_2 = 0$, i.e.~$r_2 \ll 1$ is a perturbation parameter, and in chart $\mathcal K_3$ we obtain the equations
		\begin{equation}
			\label{eq:K3_eqns}
			\begin{split}
				\partial_{t_3} \bm \psi_3 &= r_3^{-\beta} \left(\M + \Lop \right) \bm \psi_3 - \frac{1}{2} \bm \psi_3 \eps_3 + r_3^{-1-\beta} \N(r_3 \bm \psi_3, r_3^2, r_3^{2+\beta}\eps_3 ) , \\
				\dot r_3 &= \frac{1}{2} r_3 \eps_3 , \\
				\dot \eps_3 &= - \frac{2 + \beta}{2} \eps_3^2 ,
			\end{split}
		\end{equation}
		where by a slight abuse of notation the overdot notation denotes differentiation with respect to the desingularized time $t_i$ in chart $\mathcal K_i$, $i=1,2,3$.
		
		\begin{remark}
			The equations for $\dot r_i$ in \eqref{eq:K1_eqns}, \eqref{eq:K2_eqns} and \eqref{eq:K3_eqns} show that $\partial_{\bar t} r = O(r)$ as $r \to 0$. It follows that the term $r^{-1} \partial_{\bar t} r$ in \eqref{eq:psi_eqn} is $O(1)$ with respect to $r \to 0$. We shall frequently use this fact in the formal calculations which follow.
		\end{remark}
		
		The change of coordinates transformations which allow one to move from chart $\mathcal K_1$ to $\mathcal K_2$ and from chart $\mathcal K_2$ to $\mathcal K_3$ are given by
		\begin{equation}
			\label{eq:kappa_ij}
			\begin{aligned}
				\kappa_{12} : \ & \bm \psi_1 = \frac{\bm \psi_2}{\sqrt{-\mu_2}} , && r_1 = r_2 \sqrt{-\mu_2} , && \eps_1 = \frac{1}{(-\mu_2)^{(2+\beta) / 2}} , && \mu_2 < 0 , \\ 
				\kappa_{23} : \ & \bm \psi_2 = \frac{\bm \psi_3}{\eps_3^{1 / (2 + \beta)}} , && \mu_2 = \frac{1}{\eps_3^{2 / (2 + \beta)}} , && 	r_2 = r_3 \eps_3^{1 / (2 + \beta)} , && \eps_3 > 0 ,
			\end{aligned}
		\end{equation}
		respectively. Finally, we note that the equations for $r_1(t_1)$, $\eps_1(t_1)$, $\mu_2(t_2)$, $r_3(t_3)$ and $\eps_3(t_3)$ can be solved by direct integration. We have
		\begin{equation}
			\label{eq:ode_solutions}
			\begin{split}
				r_1(t_1) &= \frac{r_1(0)}{2^{1/(2+\beta)}} \left( 2 - (2+\beta) \eps_1(0) t_1 	\right)^{1/(2+\beta)} , \\
				\eps_1(t_1) &= \frac{2 \eps_1(0)}{2 - (2 + \beta) \eps_1(0) t_1} , \\
				\mu_2(t_2) &= \mu_2(0) + t_2, \\
				r_3(t_3) &= \frac{r_3(0)}{2^{1/(2+\beta)}} \left( 2 + (2+\beta) \eps_3(0) t_3 \right)^{1/(2+\beta)} , \\
				\eps_3(t_3) &= \frac{2 \eps_3(0)}{2 + (2 + \beta) \eps_3(0) t_3} .
			\end{split}
		\end{equation}
		
		\
		
		It remains to simplify the equation for $\bm \psi$ or, equivalently, those for $\bm \psi_i$ in charts $\mathcal K_i$, $i=1,2,3$. In particular, in order to obtain a simpler `desingularized problem' on the blow-up surface, these equations should be formally well-defined as $r \to 0$.
		
		
		\subsection{Step (II): Modulation reduction via the method of multiple scales}
		\label{sub:method_of_multiple_scales}
		
		
		A key insight of modulation theory is that even in the absence of a spectral gap, the dynamics near the onset of an instability should be dominated by a relatively small (but still uncountably infinite) subset of wavenumbers about the unstable modes. This observation motivates the introduction of multiple scales approaches. Our aim in this section is to extend this approach to the fast-slow setting, by coupling the geometric blow-up method introduced in Section \ref{sub:geometric_blow_up} to the formal multi-scale perturbation method known as the \textit{method of multiple scales}. 
		%
		%
		In this context, modulation equations are derived as solutions to formal solvability conditions. We briefly outline the general idea here, up to a point, however the specifics of the relevant solvability condition can depend on the problem. We therefore defer the formulation and statement of the particular conditions for each of the model problems from Section \ref{sec:dynamic_Turing_instability} to Section \ref{sec:model_problems}.
		
		The first step is to consider $\bm \psi$ as a function of $\bar{\bm x}$, $\bar t$, as well as $\bm x$ and/or $t$ (depending on which expression from \eqref{eq:psi_options} we have). Importantly, $\bar{\bm x}, \bar t , \bm x$ and $t$ are treated as \textit{independent variables}. We then expand terms in powers of $r(\bar t)$. This is reminiscent of the `classical' method of multiple scales, except for the following additional complications:
		\begin{itemize}
			\item The usual small parameter $\delta$ has been replaced by a small but time-dependent variable $r(\bar t)$;
			\item As a consequence of the above, simple rescalings of time and space have been replaced by time-dependent desingularizations defined by \eqref{eq:general_desingularization};
			\item We work in a non-trivial geometry (the blown-up space). 
		\end{itemize}
		Nevertheless, we may proceed as usual and account for additional complications as they arise.
		
		
		Since $\bar{\bm x}, \bar t , \bm x$ and $t$ are treated as independent variables, we need to replace partial derivatives as follows:
		\begin{equation}
			\label{eq:partial_derivatives}
			\partial_{x_j} \mapsto  
			\begin{cases}
				\partial_{x_j} + r \partial_{\bar x_j} , & j = 1, \ldots, p, \\
				\partial_{x_j} , & j = p+1, \ldots, n , 
			\end{cases}
			\qquad 
			\partial_t \mapsto \partial_t + r^\beta \partial_{\bar t} .
		\end{equation}
		We now write $\bm \psi$ as an expansion
		\begin{equation}
			\label{eq:psi_series}
			\bm \psi(\cdots,\bar x,\bar t) = \sum_{k=0}^\infty \bm \psi^{(k)}(\cdots, \bar x, \bar t) r(\bar t)^k ,
		\end{equation}
		where the vector-valued functions $\bm \psi^{(k)} = (\psi^{(k)}_1, \ldots, \psi^{(k)}_N)^\transpose$ are formally $O(1)$ with respect to $r$ as $r \to 0$.
		
		\begin{remark}
			In the following we use superscripts to denote constant or functional coefficients in expansions, and subscripts to specify vector components and/or local coordinates in $\mathcal K_i$.
		\end{remark}
		
		We also need expansions for $\Lop$ and $\N$:
		\[
		\Lop \mapsto \mathcal L = \sum_{k=0}^{m} \mathcal L^{(k)} r^k , \qquad
		\N \mapsto \mathcal N = r \sum_{k=0}^\infty \mathcal N^{(k)} r^k ,
		\]
		where the total number of terms in the $\mathcal L$ expansion is determined by the order $m$ of the differential operator $\Lop$. The expansion for $\N$ will not be specified in greater detail, since we have not assumed a particular structure or form beyond `sufficient regularity' and the requirement \eqref{eq:N_cond_eps}. These two properties suffice to explain why the series for $\N$ begins at $O(r)$, however, since
		\[
		\N(\bm 0, \mu, 0) = 0 \ \implies \ 
		\lim_{r \to 0} \N(r \bm \psi, r^2 \bar \mu, r^{2 + \beta} \bar \eps) = \N(\bm 0, 0, 0) = 0 .
		\]
		The operators $\mathcal L^{(k)}$ defining $\mathcal L$ are described in the following.
		
		\begin{lemma}
			\label{lem:L}
			The components of $\mathcal L^{(l)} = \textup{diag} ( \mathcal L^{(l)}_1, \mathcal L^{(l)}_2, \ldots, \mathcal L_N^{(l)} )$ are given by
			\begin{equation}
				\label{eq:mathcal_Lj}
				\mathcal L_j^{(l)} = \sum_{|\alpha| \leq m} a_{j,\alpha}(\bm x) \mathcal D_x^{(\alpha,l)} ,
			\end{equation}
			where
			\[
			\mathcal D_x^{(\alpha,l)} = 
			\begin{cases}
				\left( \sum_{\substack{q_1, q_2, \ldots , q_p \geq 0 \\ q_1 + q_2 + \cdots + q_p = l}} \prod_{k=1}^p C_{\alpha_k q_k} \left( \partial_{x_k}, \partial_{\bar x_k} \right) \right) \prod_{s=p+1}^n \partial_{x_s}^{\alpha_s} , & p \in \{1, \ldots, n-1\}, \\
				\sum_{\substack{q_1, q_2, \ldots , q_p \geq 0 \\ q_1 + q_2 + \cdots + q_p = l}} \prod_{k=1}^p C_{\alpha_k q_k} \left( \partial_{x_k}, \partial_{\bar x_k} \right) , & p = n, \\
			\end{cases}
			\]
			and
			\begin{equation}
				\label{eq:Calphak}
				C_{\alpha_k q_k}\left( \partial_{x_k}, \partial_{\bar x_k} \right) =
				\begin{cases}
					{\alpha_k \choose q_k} \partial_{x_k}^{\alpha_k - q_k} \partial_{\bar x_k}^{q_k} ,
					& q_k \leq \alpha_k, \\
					0 & q_k >  \alpha_k ,
				\end{cases}
			\end{equation}
			for each $k = 1,\dots, p$. In particular $\mathcal L^{0} = \Lop$ (it only depends on $\bm x$), and 
			the components of $\mathcal L \bm \psi = ( \mathcal L_1 \psi_1, \mathcal L_2 \psi_2, \ldots, \mathcal L_N \psi_N )^\transpose$ have the form
			\begin{equation}
				\label{eq:Lj_psij}
				\mathcal L_j \psi_j = \sum_{s \geq 0} \left( \sum_{q=0}^s \mathcal L_j^{(q)} \psi^{(s-q)}_j \right) r^s .
			\end{equation}
		\end{lemma}
		
		\begin{proof}
			This can be shown directly using the definition of $\Lop$ in \eqref{eq:L} and the expressions in \eqref{eq:partial_derivatives}. See Appendix \ref{app:proof_of_lemma_L} for details.
		\end{proof}

		
		Substituting the power series expressions for $\bm \psi$, $\Lop$, $\N$ into \eqref{eq:psi_eqn}, replacing partial derivatives as in \eqref{eq:partial_derivatives}, and matching terms with equal powers of $r$ leads to the following recursive set of equations:
		\begin{equation}
			\label{eq:matching_eqn}
			\left( \partial_t - \M - \mathcal L^{(0)} \right) \bm \psi^{(\nu)} = \bm B^{(\nu)} , \qquad 
			\nu = 0, 1, \ldots ,
		\end{equation}
		where the first $\beta$ orders are given by
		\begin{equation}
			\begin{split}
				\bm B^{(0)} &= 0, \\
				\bm B^{(\nu)} &= \sum_{q = 1}^\nu \mathcal L^{(q)} \bm \psi^{(\nu - q)} + \mathcal N^{(\nu)} , \qquad \nu = 1, \ldots, \beta - 1, \\
				\bm B^{(\beta)} &= - \left( \partial_{\bar t} + r^{-1} \partial_{\bar t} r \right) \bm \psi^{(0)} + \sum_{q = 1}^\beta \mathcal L^{(q)} \bm \psi^{(\beta - q)} + \mathcal N^{(\beta)} .
			\end{split}
		\end{equation}
		Equation \eqref{eq:matching_eqn} can be solved recursively for the $\bm \psi^{(\nu)}$, since the $\nu = 0$ equation is homogeneous and the $\bm B^{(\nu)}$ with $\nu \geq 1$ only depend on lower order functions $\bm \psi^{(k)}$ with $k < \nu$. The aim is to obtain a closed form equation for the leading order approximation $\bm \psi^{(0)}$. The general form of the solution follows after solving the homogeneous equation for $\nu = 0$, but a higher order solvability condition at $\nu = \beta$ needs to be imposed in order to pin down the dependence on the small space and time scales $\bar{\bm x}$ and $\bar t$, which do not appear in $\partial_t - \M - \mathcal L^{(0)}$. Thus, in order to obtain a modulation equation, one must
		\begin{enumerate}
			\item[(i)] Solve equations \eqref{eq:matching_eqn} recursively for $\bm \psi^{(\nu)}$ with $\nu = 0, \ldots, \beta - 1$, and
			\item[(ii)] Impose a solvability condition at $\nu = \beta$.
		\end{enumerate}
		If (i) can be achieved, then the formal limit as $r \to 0$ is well-defined in equation \eqref{eq:psi_eqn}, and therefore also in \eqref{eq:K1_eqns}, \eqref{eq:K2_eqns} and \eqref{eq:K3_eqns}. More explicitly, if solutions $\bm \psi^{(\nu)}$ to \eqref{eq:matching_eqn} exist for all $\nu = 0,1,\dots,\beta-1$, then equation \eqref{eq:psi_eqn} can be written as
		\[
		\begin{split}
			(\partial_t - \M - \mathcal L^{(0)}) \sum_{k=\beta}^\infty \bm \psi^{(k)} r^{k-\beta} &=
			- r^{-1} \left( \sum_{k=0}^\infty \bm \psi^{(k)} r^k \right) \partial_{\bar t} r \\
			&+ \sum_{s \geq \beta} \sum_{q=0}^s \mathcal L^{(q)} \bm \psi^{(s - q)} r^{s-\beta} + \sum_{k=\beta}^\infty \mathcal N^{(k)} r^{k-\beta} ,
		\end{split}
		\]
		which has the following well-defined and non-trivial limit as $r \to 0$:
		\begin{equation}
			\label{eq:psi_r0}
			\partial_{\bar t} \bm \psi^{(0)} = - r^{-1}  \bm \psi^{(0)} \partial_{\bar t} r - \left( \partial_t - \M - \mathcal L^{(0)} \right) \bm \psi^{(\beta)} + \sum_{q = 1}^\beta \mathcal L^{(q)} \bm \psi^{(\beta - q)} + \mathcal N^{(\beta)} .
		\end{equation}
		Equation \eqref{eq:psi_r0} describes the `desingularized' dynamics on the blow-up manifold, i.e.~after restriction to the invariant subspace $\{r = 0\}$.
		
		\begin{remark}
			Due to the recursive structure of the equations in \eqref{eq:matching_eqn}, it is typical for solutions to $\bm \psi^{(\nu)}$ with $\nu \geq 1$ to exist as long as solutions $\bm \psi^{(0)}$ to the homogeneous $\nu = 0$ problem exist. Note however that the recursive structure leads to a trade-off between formal accuracy and the degree of spatial regularity. Specifically, the degree of spatial regularity of $\bm \psi^{(0)}$ required in order to make the formal approximation $\bm \psi \approx \sum_{k=0}^l \bm \psi^{(k)} r^k$ will in general increase with $l$.
		\end{remark} 
		
		Notice that the right-hand side in equation \eqref{eq:psi_r0} (as well as higher order variants of this equation) may, in general, still depend on $\bm x$, $t$, $\bar{\bm x}$ and $\bar t$. However, large and small scales are naturally separated via the form of solutions to the base homogeneous equation for $\nu = 0$, which depend explicitly on $\bm x$, $t$, but only implicitly on $\bar{\bm x}$, $\bar t$ via unknown functions $A(\bar{\bm x}, \bar t) \in \mathbb C$ introduced by integration. The aim in Step (ii) is to obtain closed form (modulation) equations for these functions via the application of a suitable solvability condition. The particular form of the solvability condition depends on the problem. However, it can often be identified using established methods and techniques, e.g.~via an application of the Fredholm theorem. In the next section we formulate and impose solvability conditions in order to obtain modulation equations for the model problems (M1)-(M4).

		\section{Modulation equations for the model problems}
		\label{sec:model_problems}
		
		We now apply the method developed in Section \ref{sec:geometric_blow-up_and_the_modulation_equations} in order to derive modulation equations for the model problems (M1)-(M4) introduced in Section \ref{sec:dynamic_Turing_instability}.

		\subsection{(M1) Modulation equations}
		\label{sec:sh}
		
		The relevant blow-up transformation in this case is presented and described in detail in \cite[Sec.~3]{Jelbart2022}. Here we restate the leading order ansatz only, which is given by the blow-up transformation \eqref{eq:general_blow-up}-\eqref{eq:general_desingularization} with
		\begin{equation}
			\label{eq:blow-up_i}
			u(x,t) = r(\bar t) \psi(x, \bar x, \bar t)  
		\end{equation}
		and $\beta = 2$. Substituting \eqref{eq:blow-up_i} into the homogeneous equation for $\nu = 0$, which in this case is given by
		\begin{equation}
			\label{eq:nu0_eqn_sh}
			\mathcal L^{(0)} \psi^{(0)}(x, \bar x, \bar t) = 0 ,
		\end{equation}
		yields the neutral solution
		\[
		\psi^{(0)}(x,\bar x, \bar t) = A(\bar x, \bar t) \me^{ix} + c.c.,
		\]
		for a (presently unknown) modulation function $A(\bar x, \bar t) \in \mathbb C$. In order to derive a modulation equation, we need a suitable solvability condition for the $\nu = \beta = 2$ equation in \eqref{eq:matching_eqn}. This can be obtained by applying the \textit{Fredholm alternative}, which implies that the equation has a solution if and only if 
		\begin{equation}
			\label{eq:solvability_sh_1}
			\int_0^{2 \pi} B^{(2)}(x,\bar x, \bar t) \me^{-ix} dx = 0.
		\end{equation}
		This condition can be simplified using the recursive structure of the equations \eqref{eq:matching_eqn}, which implies that $B^{(2)}$ can be written as
		\[
		B^{(2)}(x, \bar x, \bar t) = \sum_{k \in \mathbb Z} B^{(2,k)} (\bar x, \bar t) \me^{ikx} .
		\]
		Thus, the solvability condition \eqref{eq:solvability_sh_1} amounts to the requirement that the $\me^{ix}$ coefficient vanishes, i.e.~
		\begin{equation}
			\label{eq:solvability_sh}
			B^{(2,1)}(\bar x, \bar t) \equiv 0.
		\end{equation}
		%
		%
		%
		Imposing \eqref{eq:solvability_sh} leads to the following result.
		
		\begin{thm}
			\label{thm:modulation_equations_sh}
			Consider the system obtained by applying the blow-up transformation defined by \eqref{eq:general_blow-up}-\eqref{eq:general_desingularization} and \eqref{eq:blow-up_i} with $\beta = 2$ to the Swift-Hohenberg system \eqref{eq:dynamic_sh}. The solvability condition \eqref{eq:solvability_sh} is satisfied if and only if $A(\bar x,\bar t) \in \mathbb C$ satisfies the following modulation equation of real GL type:
			\begin{equation}
				\label{eq:modulation_eqn_global_i}
				\partial_{\bar t} A = 4 \partial_{\bar x}^2 A + \left( \bar \mu(\bar t) - r(\bar t)^{-1} \partial_{\bar t} r(\bar t) \right) A - 3 A |A|^2 .
			\end{equation}
			%
			This leads to the following modulation equations in charts $\mathcal K_l$:
			\begin{equation}
				\label{eq:modulation_eqn_i_charts}
				\begin{aligned}
					\mathcal K_1 : \ \ & \partial_{t_1} A_1 = 4 \partial_{x_1}^2 A_1 + \left(-1 + \frac{\eps_1(t_1)}{2} \right) A_1 - 3 A_1 |A_1|^2 , \\
					\mathcal K_2 : \ \ & \partial_{t_2} A_2 = 4 \partial_{x_2}^2 A_2 + \mu_2(t_2) A_2 - 3 A_2 |A_2|^2 , \\
					\mathcal K_3 : \ \ & \partial_{t_3} A_3 = 4 \partial_{x_3}^2 A_3 + \left(1 - \frac{\eps_3(t_3)}{2} \right) A_3 - 3 A_3 |A_3|^2 ,
				\end{aligned}
			\end{equation}
			where $\eps_1(t_1)$, $\mu_2(t_2)$ and $\eps_3(t_3)$ are given by \eqref{eq:ode_solutions}.
		\end{thm}
		
		\begin{proof}
			We need to (i) solve equations \eqref{eq:matching_eqn} recursively for $\nu = 0, 1$, and (ii) apply the solvability condition \eqref{eq:solvability_sh} at $\nu = \beta = 2$. 
			
			\
			
			The solution at $\nu = 0$ is already described above.
			At $\nu = 1$ we have
			\[
			B^{(1)} = \mathcal L^{(1)} \psi^{(0)} 
			= - 4 \partial_{\bar x} \partial_x (1 + \partial_x^2) \psi^{(0)} = 0 ,
			\]
			where we used the fact that
			\[
			\N(r \psi, r^2 \bar \mu, r^4 \bar \eps) = r^3 (\bar \mu \psi - \psi^3) = O(r^3) \ \implies \ \mathcal N^{(1)} = 0 ,
			\]
			together with the expression for $\mathcal L^{(1)}$ obtained using Lemma \ref{lem:L}, and the fact that $(1 + \partial_x^2) \psi^{(0)} = 0$ due to the $\nu = 0$ equation \eqref{eq:nu0_eqn_sh}. It follows that
			\[
			\psi^{(1)}(x, \bar x, \bar t) = A^{(1)}(\bar x, \bar t) \me^{ix} + c.c..
			\]
			
			\
			
			It remains to apply the solvability condition \eqref{eq:solvability_sh} at $\nu = 2$. We have
			\begin{equation}
				\label{eq:B2_sh}
				B^{(2)} = -(\partial_{\bar t} + r^{-1} \partial_{\bar t} r) \psi^{(0)} + \mathcal L^{(1)} \psi^{(1)} + \mathcal L^{(2)} \psi^{(0)} + \mathcal N^{(2)} ,
			\end{equation}
			where $\mathcal L^{(1)}, \psi^{(0)}, \psi^{(1)}$ are as above and
			\[
			\mathcal L^{(2)} = - 2 \partial_{\bar x}^2 (1 + 3 \partial_x^2) , \qquad 
			\mathcal N^{(2)} = \bar \mu \psi^{(0)} - {\psi^{(0)}}^3 .
			\]
			Substituting these expressions into \eqref{eq:B2_sh} and imposing the solvability condition \eqref{eq:solvability_sh} yields the dynamic GL equation \eqref{eq:modulation_eqn_global_i}. The modulation equations in charts $\mathcal K_l$ are obtained directly from \eqref{eq:modulation_eqn_global_i} using the local coordinate formulae \eqref{eq:coordinates}.
		\end{proof}
		
		Theorem \ref{thm:modulation_equations_sh} implies that solutions $u(x, t)$ to the Swift-Hohenberg problem \eqref{eq:dynamic_sh} are formally approximated by
		\[
		u(x,t) = r \left( A(\bar x,\bar t) \me^{i x} + c.c. \right) + O(r^2) ,
		\]
		where $A(\bar x, \bar t) \in \mathbb C$ satisfies the real GL equation \eqref{eq:modulation_eqn_global_i}. It is interesting to compare the modulation equation obtained in Theorem \ref{thm:modulation_equations_sh} with the well-known modulation equation \eqref{eq:modulation_eqn_static_sh} obtained in static modulation theory. Like \eqref{eq:modulation_eqn_static_sh}, equation \eqref{eq:modulation_eqn_global_i} is a real GL equation. However, equation \eqref{eq:modulation_eqn_global_i} has a time-dependent linear coefficient $\mu(\bar t) - r(\bar t)^{-1} \partial_{\bar t} r(\bar t)$, which is a consequence of the slow parameter drift induced by $\dot \mu = \eps$. Moreover, equation \eqref{eq:modulation_eqn_global_i} is posed in a non-trivial geometry (the blown-up space), and it depends on desingularized (as opposed to rescaled) space and time variables $\bar x$ and $\bar t$.
		
		The role of the time-dependent linear coefficient becomes clearer upon inspection of the equations in charts $\mathcal K_l$, see \eqref{eq:modulation_eqn_i_charts}. One can show that the functions $\eps_1(t_1)$ and $\eps_3(t_3)$ remain small and bounded for long enough time-scales to ensure that the linear coefficients in bounded subsets of $\mathcal K_1$ and $\mathcal K_3$ are negative and positive, respectively, whereas $\mu_2(t_2)$ changes sign in chart $\mathcal K_2$. Thus on the linear level, solutions which start with initial conditions in $\mu < 0$ are exponentially contracted towards $A_1 = 0$ in chart $\mathcal K_1$. Exponential contraction is lost in chart $\mathcal K_2$ as $\mu(t_2)$ becomes positive, and solutions in chart $\mathcal K_3$ are exponentially repelled from $A_3 = 0$. We refer again to \cite{Jelbart2022} for a more detailed treatment of the dynamics. 
		
		\begin{remark}
			Equations \eqref{eq:modulation_eqn_i_charts} can also be formulated as autonomous systems of equations by using the relevant expressions in \eqref{eq:K1_eqns}, \eqref{eq:K2_eqns} and \eqref{eq:K3_eqns}. For example, the equations in $\mathcal K_1$ are
			\begin{equation}
				\label{eq:K1_sh}
				\begin{split}
					\partial_{t_1} A_1 &= 4 \partial_{x_1}^2 A_1 + \left(-1 + \frac{\eps_1}{2} \right) A_1 - 3 A_1 |A_1|^2 , \\
					\dot r_1 &= - \frac{1}{2} r_1 \eps_1 , \\
					\dot \eps_1 &= \frac{2 + \beta}{2} \eps_1^2 .
			\end{split}
			\end{equation}
			This formulation is more amenable to dynamical systems approaches. Similar observations apply to all of the modulation equations derived in this work.
		\end{remark}
		
		\begin{remark}
			Higher order corrections to $\psi$ can be obtained recursively via the equations \eqref{eq:matching_eqn} and the successive application of solvability conditions. In the Turing instability case, the relevant solvability conditions at higher orders amount to the requirement that $\me^{i \xi x}$ coefficients vanish at the non-critical frequencies $\xi \in \mathbb Z \setminus \{ \pm 1\}$. Higher order corrections can also be implemented directly in the blow-up transformation as in \cite{Jelbart2022}.
		\end{remark}
		
		\begin{remark}
			\label{rem:hyperbolicity}
			From a dynamical point of view, it is significant to note that the the blown-up equations \eqref{eq:modulation_eqn_i_charts} have improved hyperbolicity properties in comparison to the original Swift-Hohenberg problem \eqref{eq:dynamic_sh}. Consider $\mathcal K_1$ equations \eqref{eq:K1_sh}. Linearising about the set of trivial steady-states $\mathcal C_1 = \{A_1 = 0, \eps_1 = 0, r_1 \geq 0 \}$ shows that the spectrum is comprised of two identically zero eigenvalues $\lambda_{r_1} = \lambda_{\eps_1} = 0$ and a continuous curve/line $\lambda_{A_1} = -1 - 4 \xi^2$, $\xi \in \R$, which is bounded in the left-half plane. Thus, there is a spectral gap, which already shows that the dynamical stability of the steady states is governed by the leading point spectrum. 
		\end{remark}
		
		\subsection{(M2) Modulation equations}
		
		
		The relevant blow-up transformation in this case is given by \eqref{eq:general_blow-up}-\eqref{eq:general_desingularization} with
		\begin{equation}
			\label{eq:blow-up_iv}
			\bm u(\bm x, t) = r(\bar t) \bm \psi(t, \bar {\bm x}, \bar t) ,
		\end{equation}
		and $\beta = 2$. Using the expressions in \eqref{eq:LM_b}, it follows that the $\nu = 0$ equation is
		\[
		\left( \partial_t - \M \right) \bm \psi^{(0)}(t, \bar {\bm x}, \bar t) = \bm 0,
		\]
		which has neutral solutions
		\[
		\bm \psi^{(0)}(t, \bar{\bm x}, \bar t) = A(\bar{\bm x},\bar t) \bm \varphi \me^{i a t} + c.c.,
		\]
		where $\bm \varphi = (1, -1 + i a^{-1})^\transpose$ is the eigenvector corresponding to the eigenvalue $\lambda_1(0,0) = i a$ at criticality. In this case, the solvability condition at $\nu = 2$ is
		\begin{equation}
			\label{eq:solvability_b_1}
			\int_0^{2 \pi / a} \bm \varphi^\ast \cdot \bm B^{(2)} (t, \bar{\bm x}, \bar t) \me^{-iat} dt = 0 ,
		\end{equation}
		where $\bm \varphi^\ast = \tfrac{1}{2} (1 - i a, -i a)^\transpose$. This can be simplified using the recursive structure of the equations \eqref{eq:matching_eqn}. Decomposing $\bm B^{(2)}$ into harmonics such that
		\[
		\bm B^{(2)}(t, \bar{\bm x}, \bar t) = \sum_{k \in \mathbb Z} \bm B^{(2,k)} (\bar{\bm x}, \bar t) \me^{i k a t} ,
		\]
		leads to the simplified solvability condition
		\begin{equation}
			\label{eq:solvability_b}
			\bm \varphi^\ast \cdot \bm B^{(2,1)}(\bar{\bm x}, \bar t) \equiv 0 .
		\end{equation}
		Imposing \eqref{eq:solvability_b} leads to the following result.
		
		\begin{thm}
			\label{thm:modulation_equations_b}
			Consider the system obtained by applying the blow-up transformation defined by \eqref{eq:general_blow-up}-\eqref{eq:general_desingularization} and \eqref{eq:blow-up_iv} with $\beta = 2$ to the Brusselator system \eqref{eq:dynamic_b_2}. The solvability condition \eqref{eq:solvability_b} is satisfied if and only if $A(\bar{\bm x}, \bar t) \in \mathbb C$ satisfies the following modulation equation of complex GL type:
			\begin{equation}
				\label{eq:modulation_eqn_global_iv}
				\partial_{\bar t} A = c_1 \Delta_{\bar {\bm x}} A + \left( c_2 \bar \mu(\bar t) - r(\bar t)^{-1} \partial_{\bar t} r(\bar t) \right) A - c_3 A |A|^2 ,
			\end{equation}
			where $\Delta_{\bar{\bm x}} := \sum_{j=1}^n \partial_{\bar x_j}^2$ and $c_1, c_2, c_3 \in \mathbb C$ are given by \eqref{eq:nu_i}. This leads to the following modulation equations in charts $\mathcal K_l$:
			\begin{equation}
				\label{eq:modulation_eqn_i_charts_iv}
				\begin{aligned}
					\mathcal K_1 : \ \ & \partial_{t_1} A_1 = c_1 \Delta_{\bm x_1} A_1 + \left(-c_2 + \frac{\eps_1(t_1)}{2} \right) A_1 - c_3 A_1 |A_1|^2 , \\
					\mathcal K_2 : \ \ & \partial_{t_2} A_2 = c_1 \Delta_{\bm x_2} A_2 + c_2 \mu_2(t_2) A_2 - c_3 A_2 |A_2|^2 , \\
					\mathcal K_3 : \ \ & \partial_{t_3} A_3 = c_1 \Delta_{\bm x_3} A_3 + \left(c_2 - \frac{\eps_3(t_3)}{2} \right) A_3 - c_3 A_3 |A_3|^2 ,
				\end{aligned}
			\end{equation}
			where $\Delta_{\bar{\bm x}_l} := \sum_{j=1}^n \partial_{\bar x_{j,l}}^2$, and $\eps_1(t_1)$, $\mu_2(t_2)$ and $\eps_3(t_3)$ are given by \eqref{eq:ode_solutions}.
		\end{thm}
		
		\begin{proof}
			Similarly to the proof of Theorem \ref{thm:modulation_equations_sh}, we need to (i) solve equations \eqref{eq:matching_eqn} up to $\nu = 1$, and (ii) apply the solvability condition \eqref{eq:solvability_b} at $\nu = \beta = 2$.
			
			\
			
			Since we are interested in solutions of the form $\psi(t,\bar{\bm x},\bar t)$ and the equations \eqref{eq:dynamic_b_2} are translation invariant with respect to the original spacial variable $\bm x$, it suffices to make the substitution
			\[
			\partial_{x_j}^2 \mapsto r^2 \partial_{\bar x_j}^2 , \qquad 
			j = 1, \ldots, n,
			\]
			which simplifies the expressions for $\mathcal L^{(k)}$ in Lemma \ref{lem:L} significantly. In particular, $\mathcal L^{(1)} = \mathbb O_{2,2}$ so that the $\nu = 1$ equation is
			\begin{equation}
				\label{eq:nu1_b}
				\left( \partial_t - \M \right) \bm \psi^{(1)} = 
				\mathcal N^{(1)} ,
			\end{equation}
			where a direct calculation shows that
			\[
			\mathcal N^{(1)} = 
			\left(\frac{1 + a^2}{a^2} {\psi_1^{(0)}}^2 + 2 a \psi_1^{(0)} \psi_2^{(0)} \right)
			\begin{pmatrix}
				1 \\
				-1
			\end{pmatrix} .
			\]
			Solving \eqref{eq:nu1_b} leads to
			\[
			\bm \psi^{(1)} = A^2 \bm V \me^{2 i a t} + \overline{A}^2 \overline{\bm V} \me^{- 2 i a t} + |A|^2 \bm V_0 + C \bm \psi^{(0)} ,
			\]
			where
			\[
			\bm V = \frac{(1 + ia)^3}{3 a^3}
			\begin{pmatrix}
				- 2 i a \\
				1 + 2 i a
			\end{pmatrix} ,
			\qquad 
			\bm V_0 = \frac{2 (a^2 - 1)}{a^3} 
			\begin{pmatrix}
				0 \\
				1
			\end{pmatrix} ,
			\]
			and $C$ is an arbitrary constant (we will not need to know it explicitly); see \cite[App.~B]{Kuramoto1984} for details on the corresponding calculations in the static setting, which are analogous for $\nu = 0$ and $\nu = 1$.
			
			\
			
			It remains to apply the solvability condition \eqref{eq:solvability_b} with
			\[
			\bm B^{(2)} = - \left( \partial_{\bar t} + r^{-1} \partial_{\bar t} r \right) \bm \psi^{(0)} + \mathcal N^{(2)} ,
			\]
			where
			\[
			\begin{split}
				\mathcal N^{(2)} &=
				\begin{pmatrix}
					d_1 \Delta & 0 \\
					0 & d_2 \Delta 	
				\end{pmatrix} 
				\bm \psi^{(0)} \\ 
				&+ \left( (1 + a^2) \bar \mu \psi_1^{(0)} + 2 \frac{1 + a^2}{a} \psi_1^{(0)} \psi_1^{(1)} + 2 a (\psi_1^{(0)} \psi_2^{(1)} + \psi_1^{(1)} \psi_2^{(0)} ) \right)
				\begin{pmatrix}
					1 \\
					-1
				\end{pmatrix} .
			\end{split}
			\]
			Substituting all of the relevant expressions above into \eqref{eq:solvability_b} yields equation \eqref{eq:modulation_eqn_global_iv}. The modulation equations in charts $\mathcal K_l$ are obtained directly from the global modulation equation \eqref{eq:modulation_eqn_global_iv} using the coordinate formulae \eqref{eq:coordinates}.
		\end{proof}
		
		Theorem \ref{thm:modulation_equations_b} implies that solutions $\bm u(\bm x, t)$ to the Brusselator system \eqref{eq:dynamic_b_2} are formally approximated by
		\[
		\bm u(\bm x,t) = r \left( A(\bar{\bm x},\bar t) \bm \varphi \me^{i a t} + c.c. \right) + O(r^2) ,
		\]
		where $A(\bar{\bm x}, \bar t) \in \mathbb C$ satisfies the complex GL equation \eqref{eq:modulation_eqn_global_iv}. 
		%
		%
		%
		Comparing the static and dynamic complex GL equations \eqref{eq:modulation_eqn_static_b} and \eqref{eq:modulation_eqn_global_iv}, we find that the same observations which were made following the proof of Theorem \ref{thm:modulation_equations_sh} in Section \ref{sec:sh} apply here as well, except that the equations are of complex (as opposed to real) GL type.
		
		\begin{remark}
			Higher order corrections can be obtained recursively via the equations \eqref{eq:matching_eqn} and the successive application of solvability conditions which amount to the requirement that $\me^{i \xi a t}$ coefficients vanish at the non-critical harmonics $\xi \in \mathbb Z \setminus \{ \pm 1\}$.
		\end{remark}
		
		\begin{remark}
			\label{rem:hyperbolicity_2}
			An analogue of Remark \ref{rem:hyperbolicity} applies to the comparison between the hyperbolicity properties of the Brusselator system \eqref{eq:dynamic_b_2} and modulation equations in charts $\mathcal K_l$ defined in \eqref{eq:modulation_eqn_i_charts_iv}. 
		\end{remark}

		\subsection{(M3) Modulation equations}
		
		The relevant blow-up transformation in this case is given by \eqref{eq:general_blow-up}-\eqref{eq:general_desingularization} with
		\begin{equation}
			\label{eq:blow-up_ii}
			\bm u(x, t) = r(\bar t) \bm \psi(x, t, \bar x, \bar t) ,
		\end{equation}
		and $\beta = 2$. Since $\M = \mathbb O_{2,2}$ (recall \eqref{eq:LM_ks}), the $\nu =0$ equation is
		\[
		(\partial_t - \mathcal L^{(0)}) \bm \psi^{(0)} = \bm 0 ,
		\]
		which has neutral solutions of the form
		\[
		\bm \psi^{(0)}(x, t, \bar x, \bar t) = 
		\begin{pmatrix}
			A_1(\bar x, \bar t) \me^{i(x - t)} + c.c. \\
			A_2(\bar x, \bar t) \me^{i(x + t)} + c.c.
		\end{pmatrix} ,
		\]
		for modulation functions $A_1(\bar x, \bar t), A_2(\bar x, \bar t) \in \mathbb C$. In order to formulate solvability conditions for $A_1$ and $A_2$, we note that the component functions in $\bm B^{(2)} = (B_1^{(2)}, B_2^{(2)})^\transpose$ can be written in the form
		\[
		B_l^{(2)}(x,t,\bar x,\bar t) = \sum_{k_1, k_2 \in \mathbb Z} B_l^{(2,k_1,k_2)}(\bar x, \bar t) \me^{i (k_1 x + k_2 t)} , \qquad 
		l = 1, 2.
		\]
		We will impose the following solvability conditions for $A_1$ and $A_2$ respectively (see however Remark \ref{rem:averaging_2} below):
		\begin{equation}
			\label{eq:solvability_ks}
			B_1^{(2,1,-1)}(\bar x, \bar t) \equiv 0, \qquad 
			B_2^{(2,1,1)}(\bar x, \bar t) \equiv 0.
		\end{equation}
		%
		%
		This leads to the following result.
		
		\begin{thm}
			\label{thm:modulation_equations_ks}
			Consider the system obtained by applying the blow-up transformation defined by \eqref{eq:general_blow-up}-\eqref{eq:general_desingularization} and \eqref{eq:blow-up_ii} with $\beta = 2$ to the Kuramoto-Sivashinsky system \eqref{eq:dynamic_ks}. The solvability condition \eqref{eq:solvability_ks} is satisfied if and only if the functions $A_1(\bar x, \bar t), A_2(\bar x, \bar t) \in \mathbb C$ obey the following system of coupled generalised complex GL equations:
			\begin{equation}
				\label{eq:modulation_eqn_global_ii}
				\begin{split}
					\partial_{\bar t} A_1 &= - \partial_{\bar x} A_1 +
					4 \partial_{\bar x}^2 A_1 + (\bar \mu(\bar t) - r(\bar t)^{-1} \partial_{\bar t} r(\bar t)) A_1 - \gamma_1 |A_1|^2 A_1 - \gamma_2 A_1 |A_2|^2, \\
					\partial_{\bar t} A_2 &= \partial_{\bar x} A_2 +
					4 \partial_{\bar x}^2 A_2 + (\bar \mu(\bar t) - r(\bar t)^{-1} \partial_{\bar t} r(\bar t)) A_2 - \gamma_1 |A_2|^2 A_2 - \gamma_2 A_2 |A_1|^2, 
				\end{split}
			\end{equation}
			where the constants $\gamma_1, \gamma_2 \in \mathbb C$ can be computed using equations \eqref{eq:matching_eqn} with $\nu = 1$. System \eqref{eq:modulation_eqn_global_ii} can alternatively be written in terms of left and right travelling wave amplitudes $A_L(\bar x_L,\bar t) := A_1(\bar x, \bar t)$ and $A_R(\bar x_R,\bar t) := A_2(\bar x, \bar t)$, where $\bar x_L := \bar x - \bar t$ and $\bar x_R := \bar x + \bar t$ respectively, as follows:
			\begin{equation}
				\label{eq:modulation_eqn_ks_2}
				\begin{split}
					\partial_{\bar t} A_L &= 
					4 \partial_{\bar x_L}^2 A_L + (\bar \mu(\bar t) - r(\bar t)^{-1} \partial_{\bar t} r(\bar t)) A_L - \gamma_1 |A_L|^2 A_L - \gamma_2 A_L |A_R|^2, \\
					\partial_{\bar t} A_R &= 
					4 \partial_{\bar x_R}^2 A_R + (\bar \mu(\bar t) - r(\bar t)^{-1} \partial_{\bar t} r(\bar t)) A_R - \gamma_1 |A_R|^2 A_R - \gamma_2 A_R |A_L|^2 .
				\end{split}
			\end{equation}
			%
			This leads to the following systems in charts $\mathcal K_1, \mathcal K_2, \mathcal K_3$ respectively:
			\[
			\begin{split}
				&
				\begin{cases}
					\partial_{t_1} A_{L,1} = 
					4 \partial_{x_{L,1}}^2 A_{L,1} + \left( -1 + \frac{\eps_1(t_1)}{2} \right) A_{L,1} - \gamma_1 |A_{L,1}|^2 A_{L,1} - \gamma_2 A_{L,1} |A_{R,1}|^2, \\
					\partial_{t_1} A_{R,1} = 
					4 \partial_{x_{R,1}}^2 A_{R,1} + \left( -1 + \frac{\eps_1(t_1)}{2} \right) A_{R,1} - \gamma_1 |A_{R,1}|^2 A_{R,1} - \gamma_2 A_{R,1} |A_{L,1}|^2 ,
				\end{cases} \\
				&
				\begin{cases}
					\partial_{t_2} A_{L,2} = 
					4 \partial_{x_{L,2}}^2 A_{L,2} + \mu_2(t_2) A_{L,2} - \gamma_1 |A_{L,2}|^2 A_{L,2} - \gamma_2 A_{L,2} |A_{R,2}|^2, \\
					\partial_{t_2} A_{R,2} = 
					4 \partial_{x_{R,2}}^2 A_{R,2} + \mu_2(t_2) A_{R,2} - \gamma_1 |A_{R,2}|^2 A_{R,2} - \gamma_2 A_{R,2} |A_{L,2}|^2 ,
				\end{cases} \\
				&
				\begin{cases}
					\partial_{t_3} A_{L,3} = 
					4 \partial_{x_{L,3}}^2 A_{L,3} + \left( 1 - \frac{\eps_3(t_3)}{2} \right) A_{L,3} - \gamma_1 |A_{L,3}|^2 A_{L,3} - \gamma_2 A_{L,3} |A_{R,3}|^2, \\
					\partial_{t_3} A_{R,3} = 
					4 \partial_{x_{R,3}}^2 A_{R,3} + \left( 1 - \frac{\eps_3(t_3)}{2} \right) A_{R,3} - \gamma_1 |A_{R,3}|^2 A_{R,3} - \gamma_2 A_{R,3} |A_{L,3}|^2 ,
				\end{cases}
			\end{split}
			\]
			where $\eps_1(t_1)$, $\mu_2(t_2)$ and $\eps_3(t_3)$ are given by \eqref{eq:ode_solutions}.
		\end{thm}
		
		\begin{proof}
			At $\nu = 1$ we have
			\begin{equation}
				\label{eq:nu1_ks}
				(\partial_t - \mathcal L^{(0)}) \bm \psi^{(1)} = \bm B^{(1)} = \mathcal L^{(1)} \bm \psi^{(0)} + \mathcal N^{(1)} ,
			\end{equation}
			where
			\[
			\mathcal L^{(1)} = 
			\begin{pmatrix}
				- 1 & 0 \\
				0 & 1
			\end{pmatrix}
			\partial_{\bar x} - 4 \mathbb{I}_2 \partial_x \partial_{\bar x} (1 + \partial_x^2)  , \quad 
			\mathcal N^{(1)} = \partial_x \left( {\psi_1^{(0)}}^2 + \psi_1^{(0)} \psi_2^{(0)} + {\psi_2^{(0)}}^2 \right)
			\begin{pmatrix}
				1 \\
				1
			\end{pmatrix} .
			\]
			Solutions for $\bm \psi^{(1)} = (\psi_1^{(1)}, \psi_2^{(1)})^\transpose$ have the form
			\[
			\begin{split}
				\psi_1^{(1)} = v_0 \psi_1^{(0)} + v_1 \partial_{\bar x} A_1 \me^{i(x-t)} + v_2 A_1^2 \me^{2i(x-t)} + v_3 A_1 A_2 \me^{2ix} + v_4 A_2^2 \me^{2i(x+t)} + c.c. , \\
				\psi_2^{(1)} = \eta_0 \psi_2^{(0)} + \eta_1 \partial_{\bar x} A_2 \me^{i(x+t)} + \eta_2 A_1^2 \me^{2i(x-t)} + \eta_3 A_1 A_2 \me^{2ix} + \eta_4 A_2^2 \me^{2i(x+t)} + c.c. ,
			\end{split}
			\]
			where the coefficients $v_i, \eta_i \in \mathbb C$, $i = 1, 2, 3, 4$ can be determined by substitution the above into equation \eqref{eq:nu1_ks} if necessary.
			
			\
			
			It remains to apply the solvability condition \eqref{eq:solvability_ks}. We have
			\[
			\bm B^{(2)} = -(\partial_{\bar t} + r^{-1} \partial_{\bar t} r) \bm \psi^{(0)} + \mathcal L^{(1)} \bm \psi^{(1)} + \mathcal L^{(2)} \bm \psi^{(0)} + \mathcal N^{(2)} ,
			\]
			where $\mathcal L^{(1)}$, $\bm \psi^{(0)}$ and $\bm \psi^{(1)}$ are as given above,
			\[
			\mathcal L^{(2)} = - 2 \mathbb I_2 \partial_{\bar x}^2 (1 + 3 \partial_x^2)  ,
			\]
			and
			\[
			\begin{split}
				\mathcal N^{(2)} &= \bar \mu \bm \psi^{(0)} + \bigg( \partial_{\bar x} \left( {\psi_1^{(0)}}^2 + \psi_1^{(0)} \psi_2^{(0)} + {\psi_2^{(0)}}^2 \right) \\
				&+ \partial_x \left( \psi_1^{(0)} \psi_1^{(1)} + \psi_1^{(0)} \psi_2^{(1)} + \psi_1^{(1)} \psi_2^{(0)} + \psi_2^{(0)} \psi_2^{(1)} \right) \bigg) 
				\begin{pmatrix}
					1 \\
					1
				\end{pmatrix} .
			\end{split}
			\]
			The solvability condition \eqref{eq:solvability_ks} 
			requires that the $\me^{i(x-t)}$ and $\me^{i(x+t)}$ coefficients in the $B_1^{(2)}$ and $B_2^{(2)}$ equations vanish, respectively. Collecting like terms and imposing this constraint leads to a coupled system of generalised complex GL type equations:
			\[
			\begin{split}
				\partial_{\bar t} A + c \partial_{\bar x} A &= 
				4 \partial_{\bar x}^2 A + (\bar \mu - r^{-1} \partial_{\bar t} r) A + i(v_2 + \eta_2) |A|^2 A + i (v_3 + \eta_3) A |B|^2, \\
				\partial_{\bar t} B - c \partial_{\bar x} B &= 
				4 \partial_{\bar x}^2 B + (\bar \mu - r^{-1} \partial_{\bar t} r) B + i(v_2 + \eta_2) |B|^2 B + i (v_3 + \eta_3) B |A|^2, 
			\end{split}
			\]
			where $c = v_0 + v_1$. In fact, $c$ is the group velocity (see e.g.~\cite{Frohoff2023}), which we can calculate directly using the dispersion relation $\omega(\xi) = i - \xi -2i \xi^2 + i \xi^4$ and
			\[
			c = \partial_\xi \omega(\xi) |_{\xi = \xi_c = 1} = 1 .
			\]
			Setting $\gamma_1 := - i(v_2 + \eta_2)$ and $\gamma_2 := -i (v_3 + \eta_3)$ yields system \eqref{eq:modulation_eqn_global_ii}. The corresponding system in terms of left and right travelling wave amplitude $A_L(\bar x_L,\bar t)$ and $A_R(\bar x_R,\bar t)$ follow directly from this expression, and the equations in charts can be obtained directly using \eqref{eq:coordinates}.
		\end{proof}

		Theorem \ref{thm:modulation_equations_ks} implies that solutions $\bm u(x,t)$ to the coupled Kuramoto-Sivashinsky system \eqref{eq:dynamic_ks} are formally approximated by
		\[
		\bm u(x, t) = r
		\begin{pmatrix}
			A_1(\bar x, \bar t) \me^{i(x - t)} + c.c. \\
			A_2(\bar x, \bar t) \me^{i(x + t)} + c.c.
		\end{pmatrix} 
		+ O(r^2) ,
		\]
		where $A_1(\bar x, \bar t), A_2(\bar x, \bar t) \in \mathbb C$ satisfy the system of generalised complex GL equations \eqref{eq:modulation_eqn_global_ii}. Similar observations apply to the comparison of the static and dynamic modulation equations \eqref{eq:modulation_eqn_static_ks} and \eqref{eq:modulation_eqn_ks_2}, respectively. Specifically, the dynamic modulation equations \eqref{eq:modulation_eqn_ks_2}, although in the same general form, are distinguished by the fact that they are posed in the blown-up space, they depend on desingularized space $\bar x$ and time $\bar t$, and that the linear coefficient depends on $\bar t$.
		
		\begin{remark}
			\label{rem:averaging_2}
			In the static setting, normal form transformations and averaging theory are used in order to show that highly oscillatory terms can be neglected; recall Remark \ref{rem:averaging}. These arguments can be used in order to show that the (static) solvability condition has a form analogous to \eqref{eq:solvability_ks}. The validity of the modulation equations in Theorem \ref{thm:modulation_equations_ks} is contingent upon the success of these arguments in the fast-slow setting, which we do not attempt to reproduce in the present work.
		\end{remark}
		
		\begin{remark}
			\label{rem:hyperbolicity_3}
			The modulation equations in charts have improved hyperbolicity properties. Specifically, an analogue of Remarks \ref{rem:hyperbolicity} and \ref{rem:hyperbolicity_2} applies.
		\end{remark}
		
		\subsection{(M4) Modulation equations}
		
		The relevant blow-up transformation in this case is given by \eqref{eq:general_blow-up}-\eqref{eq:general_desingularization} with
		\begin{equation}
			\label{eq:blow-up_iii}
			\bm u(\bm x, t) = 
			\begin{pmatrix}
				\bm U'(\bm x, t) \\
				p(\bm x)
			\end{pmatrix}
			=
			r(\bar t) \bm \psi(\bar{\bm x}, \bar t) ,
		\end{equation}
		$\mu = \mathcal R'$, $\beta = 4$ and $\bar{\bm x} = (\bar x, y)$. Note that the $y$-coordinate has not been desingularized, and that we still have periodic boundary conditions for $\bm \psi$ at $y \in \partial \widetilde{\Omega} = \{\pm \pi\}$. The structure of the equations obtained after expanding in powers of $r$ will vary slightly to \eqref{eq:matching_eqn} due to the additional constraints \eqref{eq:ns_constraints}. In the present case, applying the blow-up transformation defined above to equation \eqref{eq:psi_eqn} leads to
		\begin{equation}
			\label{eq:ns_blown_up}
			\partial_{\bar t} \bm \psi = r^{-4} \left(\M + \Lop \right) \bm \psi - r^{-1} \bm \psi \partial_{\bar t} r + r^{-5} \N(r \bm \psi, r^2 \bar \mu, r^6 \bar \eps) , 
		\end{equation}
		where $\bm \psi = (\bar u, \bar v, \bar p)^\transpose$ and $\M$, $\Lop$ and $\N$ given by \eqref{eq:LM_ns}-\eqref{eq:N_ns}. We also obtain the following (blown-up) incompressibility and zero mean flow constraints
		\begin{equation}
			\label{eq:ns_constraints_blown_up}
			r \partial_{\bar x} \bar u + \partial_y \bar v = 0 , \qquad
			[\bar u]_{\widetilde \Omega} = \frac{1}{| \widetilde \Omega |} \int_{\widetilde	 \Omega} \bar u(\bar x, y) dy = 0 ,
		\end{equation}
		as a consequence of the corresponding conditions in \eqref{eq:ns_constraints}.
		
		The aim is to obtain a modulation equation which governs the leading order asymptotic approximation for the velocity field $\bm{\bar U} = (\bar u, \bar v)^\transpose$, i.e.~for the first two components of $\bm \psi$. We introduce the following componentwise notation for the series expansion for $\bm \psi = (\bar{\bm U}, \bar p)$ defined in \eqref{eq:psi_series}:
		\begin{equation}
			\label{eq:ns_psi_expansion}
			\bm{\bar U}(\bar x, y, \bar t) = 
			\sum_{j=0}^\infty \bm U^{(j)}(\bar x, y, \bar t) r(\bar t)^j , \qquad
			\bar p(\bar x, y) = r(\bar t) \sum_{j=0}^\infty p^{(j)}(\bar x, y) r(\bar t)^j , 
		\end{equation}
		cf.~the static expansions \eqref{eq:approximation_ns}. We shall also write 
		$\bm U^{(j)} = (u^{(j)}, v^{(j)})^\transpose$ for each $j \in \mathbb N$. The aim, in this notation, is to find a closed form (modulation) equation for $\bm U^{(0)}(\bar x, y, \bar t)$. Calculations in the proof of Theorem \ref{thm:modulation_equations_ns} below show that
		\begin{equation}
			\label{eq:U0_ns}
			\bm U^{(0)} (\bar x, y, \bar t) = A(\bar x, \bar t) \bm \varphi ,
		\end{equation}
		where $A(\bar x, \bar t) \in \R$ and $\bm \varphi = (- \sqrt 2 \cos y, 1)^\transpose$ is the eigenvector associated to the eigenvalue $\lambda(0,0) = 0$ at criticality. The solvability condition for $A(\bar x, \bar t)$ can be formulated in terms of the residual, which we write as a power series in $r$:
		\[
		\begin{split}
			r^{-1} \textup{Res} \left(r \bm \psi \right) &= 
			- r^{-1} \partial_t (r(\bar t) \bm \psi(\bar{\bm x}, \bar t)) - \left(\M + \Lop \right) \bm \psi(\bar{\bm x}, \bar t) + r^{-1} \N \left(r \bar{\bm U}, r \bar p, r^2 \bar{\mathcal R}, r^6 \bar \eps \right) \\
			&= \sum_{k=0}^\infty \mathcal Q^{(k)}(\bar{\bm x}, \bar t) r(\bar t)^k ,
		\end{split}
		\]
		where $\mathcal Q^{(k)} = (\mathcal Q_1^{(k)}, \mathcal Q_2^{(k)}, 0)^\transpose$ for each $k \in \mathbb N_+$ and we write $\mu = \mathcal R' = r^2 \bar \mu = r^2 \bar{\mathcal R}$. A solvability condition can be imposed at $\nu = 3$ on the $v^{(0)}$ coordinate. This is sufficient to determine $u^{(0)}$ and thus also $\bm U^{(0)}$ using \eqref{eq:U0_ns}. We require that
		\begin{equation}
			\label{eq:solvability_ns}
			\mathcal Q_2^{(3)}(\bar{\bm x}, \bar t) \equiv 0 .
		\end{equation}
		Imposing \eqref{eq:solvability_ns} leads to the following result.
		
		
		\begin{thm}
			\label{thm:modulation_equations_ns}
			Consider system \eqref{eq:ns_blown_up} with the constraints \eqref{eq:ns_constraints_blown_up}. A necessary condition for the solvability condition \eqref{eq:solvability_ns} to be satisfied is that $A(\bar x, \bar t) \in \R$ satisfies the following modulation equation of Cahn-Hilliard type:
			\begin{equation}
				\label{eq:modulation_eqn_global_iii}
				\partial_{\bar t} A = - r(\bar t)^{-1} \partial_{\bar t} r(\bar t) A - 3 \partial_{\bar x}^4 A - \sqrt 2 \bar{\mathcal R}(\bar t) \partial_{\bar x}^2 A + \frac{2}{3} \partial_{\bar x}^2 \left( A^3 \right) .
			\end{equation}
			This leads to the following modulation equations in charts $\mathcal K_l$:
			\begin{equation}
				\label{eq:modulation_eqn_i_charts_iii}
				\begin{aligned}
					\mathcal K_1 : \ \ & \partial_{t_1} A_1 = \frac{\eps_1(t_1)}{2} A_1 - 3 \partial_{x_1}^4 A_1 + \sqrt 2 \partial_{x_1}^2 A_1 + \frac{2}{3} \partial_{x_1}^2 \left( A_1^3 \right) , \\
					\mathcal K_2 : \ \ & \partial_{t_2} A_2 = - 3 \partial_{x_2}^4 A_2 - \sqrt 2 \mathcal R_2(t_2) \partial_{x_2}^2 A_2 + \frac{2}{3} \partial_{x_2}^2 \left( A_2^3 \right) , \\
					\mathcal K_3 : \ \ & \partial_{t_3} A_3 = - \frac{\eps_3(t_1)}{2} A_3 - 3 \partial_{x_3}^4 A_3 - \sqrt 2 \partial_{x_3}^2 A_3 + \frac{2}{3} \partial_{x_3}^2 \left( A_3^3 \right) ,
				\end{aligned}
			\end{equation}
			where $\eps_1(t_1)$, $\mathcal R_2(t_2)$ and $\eps_3(t_3)$ are given by \eqref{eq:ode_solutions} with $\mu_2 = \mathcal R_2$.
		\end{thm}
		
		\begin{proof}
			We first need to solve the equations
			\begin{equation}
				\label{eq:ns_matching_eqns}
				\partial_{\bar x} u^{(\nu)} + \partial_y v^{(\nu + 1)} = 0 , \qquad
				\mathcal Q_1^{(\nu)} = 0 , \qquad 
				\mathcal Q_2^{(\nu)} = 0 , 
			\end{equation}
			recursively for $\nu = 0, 1, 2$, starting with $\nu = 0$ and $\partial_y v^{(0)} = 0$ (this equation and the left-most equation above follow from the incompressibility condition in \eqref{eq:ns_constraints_blown_up}). Notice that the right-hand side in \eqref{eq:blow-up_iii} has no explicit dependence on 
			$x$, and that the right-hand side in \eqref{eq:dynamic_ns} is translation invariant in $x$. It therefore suffices to make the replacements
			\[
			\partial^l_{x} \mapsto r^l \partial^l_{\bar x}, 
			\]
			for $l \in \mathbb N_+$ accordingly. Following this, the incompressibility condition in \eqref{eq:ns_constraints_blown_up} and the equations for the individual components in \eqref{eq:ns_blown_up} can be written as
			\begin{equation}
				\label{eq:ns_matching_eqns_2}
				\begin{split}
					\partial_{y} \bar v &= - r \partial_{\bar x} \bar u , \\
					\partial_{y}^2 \bar u &= \sqrt 2 \cos y \bar v + \left( \sqrt{2} \sin y \partial_{\bar x} \bar u + \bar v \partial_y \bar u \right) r +
					\left( \bar u \partial_{\bar x} \bar u - \partial_{\bar x}^2 \bar u + \partial_{\bar x} \bar p + \bar{\mathcal R} \bar v \cos y \right) r^2 \\
					& + \bar{\mathcal R} \sin y \partial_{\bar x} \bar u r^3 +
					\left( \partial_{\bar t} \bar u + r^{-1} \bar u \partial_{\bar t} r \right) r^4 + \bar \eps \sin y r^5 , \\
					\partial_y \bar p &= - \partial^2_{\bar x y} \bar u - \sqrt 2 \sin y \partial_{\bar x} \bar v - \bar v \partial_y \bar v + 
					\left( - \bar u \partial_{\bar x} \bar v + \partial_{\bar x}^2 \bar v \right) r - \bar{\mathcal R} \sin y \partial_{\bar x} \bar v r^2 \\
					&- \left( \partial_{\bar t} \bar v + \bar v r^{-1} \partial_{\bar t} r \right) r^3 .
				\end{split}
			\end{equation}
			These equations have a very similar structure to the equations obtained in the static calculations, cf.~\cite[eqn~(2.3)]{Nepomniashchii1976}, except that we obtain additional terms with factors $r^{-1} \partial_{\bar t} r$ due to the time-dependent desingularization $\partial_t = r^4 \partial_{\bar t}$ (as opposed to the simple rescaling $T = \delta^4 t$). The idea is to obtain explicit formulae for the conditions in \eqref{eq:ns_matching_eqns} by substituting the series expansions for $\bar u$, $\bar v$ and $\bar p$ in \eqref{eq:ns_psi_expansion} into equations \eqref{eq:ns_matching_eqns_2} and matching in powers of $r$. Since the `new terms' do not appear until $O(r^3)$, however, the calculations for $\nu = 0, 1, 2$ agree with those in \cite{Nepomniashchii1976}. We therefore omit the details for $\nu = 1, 2$. The details for $\nu = 0$ will however be included, in order to justify the expression \eqref{eq:U0_ns}.
			
			
			
			At $\nu = 0$ we obtain
			\[
			\partial_y v^{(0)} = 0, \quad 
			\partial_y^2 u^{(0)} = \sqrt 2 \cos y v^{(0)} , \quad
			\partial_y p^{(0)} = - \partial^2_{\bar x y} u^{(0)} - \sqrt 2 \sin y \partial_{\bar x} v^{(0)} ,
			\]
			where we eliminated the $v^{(0)} \partial_y v^{(0)}$ term in the last equation using the first equation. Solving these equations yields
			\[
			v^{(0)} = A(\bar x, \bar t), \quad 
			u^{(0)} = - \sqrt 2 \cos y A + C_1(\bar x, \bar t), \quad 
			p^{(0)} = 2 \sqrt 2 \sin y \partial_{\bar x} A + C_2(\bar x, \bar t) ,
			\]
			for some (presently unknown) functions $A(\bar x, \bar t)$, $C_1(\bar x, \bar t)$ and $C_2(\bar x, \bar t)$. In order to obtain a leading order approximation, we need to determine an equation for $A$. This suffices to describe both $u^{(0)}$ and $v^{(0)}$, because the zero mean flow condition in \eqref{eq:ns_constraints_blown_up} gives
			\[
			\int_{-\pi}^{\pi} u^{(0)}(\bar x, y, \bar t) dy = 
			2 \pi C_1(\bar x, \bar t) = 0 \qquad 
			\implies \qquad C_1(\bar x, \bar t) = 0.
			\]
			We need to go to higher orders in order to apply the relevant solvability condition.
			
			\
			
			The modulation equation for $A$ is obtained at $\nu = 3$ via the solvability condition associated to the equation $\mathcal Q_2^{(3)} = 0$, which can be written as
			\begin{equation}
				\label{eq:p3_eqn}
				\begin{split}
					\partial_y p^{(3)} &= - \partial^2_{\bar x y} u^{(3)} - \sqrt{2} \sin y \partial_{\bar x} v^{(3)} - v^{(0)} \partial_{\bar x} v^{(2)} - 3 v^{(2)} \partial_{\bar x} v^{(1)} - v^{(3)} \partial_{\bar x} v^{(0)} \\
					&- u^{(0)} \partial_{\bar x} v^{(2)} - 2 u^{(1)} \partial_{\bar x} 	v^{(1)} - u^{(2)} \partial_{\bar x} v^{(0)} + \partial_{\bar x}^2 v^{(2)} - \bar{\mathcal R} \sin y \partial_{\bar x} v^{(1)} \\
					&- \left( \partial_{\bar t} v^{(0)} + v^{(0)} r^{-1} \partial_{\bar t} r \right) .
				\end{split}
			\end{equation}
			The functions $u^{(k)}$ and $v^{(k)}$ with $k = 1,2,3$ can be written entirely in terms of lower order terms $u^{(j)}, v^{(j)}$ and $p^{(j)}$ with $j < k$. This implies that they too have the same form as in the static derivations in \cite{Nepomniashchii1976}, since $j \leq 2$. In fact, the only difference in the form of equation \eqref{eq:p3_eqn} when compared with the corresponding static equation is that the term $\partial_{\bar t} v^{(0)} + v^{(0)} r^{-1} \partial_{\bar t} r$ appears instead of $\partial_T v^{(0)}$. We therefore obtain the same requirement on $v^{(0)}$, except that $\partial_T v^{(0)}$ is replaced by $\partial_{\bar t} v^{(0)} + v^{(0)} r^{-1} \partial_{\bar t} r$, $X$ is replaced by $\bar x$, $T$ is replaced by $\bar t$ and $\bar{\mathcal R}$ depends on $\bar t$. Namely, we obtain
			%
			\[
			\partial_{\bar t} v^{(0)} + v^{(0)} r^{-1} \partial_{\bar t} r = - 3 \partial^4_{\bar x} v^{(0)} - \sqrt 2 \bar{\mathcal R}(\bar t) \partial^2_{\bar x} v^{(0)} + \frac{2}{3} \partial^2_{\bar x} \left( {v^{(0)}}^3 \right) ,
			\]
			cf.~\cite[eqn.~(2.4)]{Nepomniashchii1976}. Substituting the expressions for $u^{(0)}$ and $v^{(0)}$ obtained above shows that the leading order approximation is
			\[
			\bar{\bm U}^{(0)}(\bar x, \bar t) = A(\bar x, \bar t) \bm \varphi ,
			\]
			where $\bm \varphi = (- \sqrt 2 \cos y, 1)^\transpose$ and $A(\bar x,\bar t) \in \R$ obeys the Cahn-Hilliard type equation \eqref{eq:modulation_eqn_global_iii}, as required. As before, the equations in charts $\mathcal K_l$ are obtained directly from equation \eqref{eq:modulation_eqn_global_iii} using the coordinate formulae \eqref{eq:coordinates}.
		\end{proof}
		
		Theorem \ref{thm:modulation_equations_ns} implies that the velocity field in system \eqref{eq:dynamic_ns} formally approximated by
		\[
		\bm U'(x,y,t) = r A(\bar x, \bar t) \bm \varphi + O(r^2) ,
		\]
		where $A(\bar x, \bar t) \in \R$ satisfies the Cahn-Hilliard equation \eqref{eq:modulation_eqn_global_iii}. In contrast to the static Cahn-Hilliard equation \eqref{eq:modulation_eqn_static_ch}, the slow parameter drift $\dot{\mathcal R} = \eps$ in \eqref{eq:dynamic_ns} leads to an additional spatially homogeneous, $A$-linear term with a time-dependent coefficient in \eqref{eq:modulation_eqn_global_iii}. This causes solutions to approach and move away from the blow-up surface in charts $\mathcal K_1$ and $\mathcal K_3$, respectively.
		
		\begin{remark}
			\label{rem:hyperbolicity_4}
			The modulation equations in charts $\mathcal K_l$ are again expected to exhibit `improved' hyperbolicity properties in comparison to the original problem \eqref{eq:dynamic_ns}. However, in contrast to the modulation equations obtained for the model problems (M1)-(M3), all of which were of GL type, the spectrum obtained after linearising around the relevant stationary states in charts $\mathcal K_1$ and $\mathcal K_3$ still have continuous curves which intersect the origin. Consider for example the spectrum obtained after linearisation along the set of steady states $\mathcal C_1 = \{A_1 = 0, \eps_1 = 0, r_1 \geq 0\}$ in chart $\mathcal K_1$. It is comprised of two zero eigenvalues and a continuous half-line due to $\lambda_{A_1} = - (\sqrt 2 + 3 \xi^2) \xi^2 \leq 0$, and is therefore only marginally stable (there is no spectral gap). This additional complication is a consequence of the incompressibility condition which, as we noted in Section \ref{sec:dynamic_Turing_instability}, can be viewed as a conservation law. A more refined account of the sense in which the blown-up equations \eqref{eq:modulation_eqn_i_charts_iii} have `improved hyperbolicity properties', and how such properties could be utilised to study the dynamics, is left for future work.
		\end{remark}

		\section{Conclusion and outlook}
		\label{sec:conclusion_and_outlook}
		
		In the absence of a rigorous theory for the existence of center or slow manifolds near dynamic bifurcations with continuous spectra, a method for obtaining formally correct modulation equations can be viewed as the first step towards a rigorous understanding of the dynamics. The primary contribution of this article was to provide a formal but systematic approach to achieving this first step, and to apply it to particular model problems featuring dynamic Turing, Hopf, Turing-Hopf and long-wave bifurcations. In order to do this, we extended and generalised the `formal part' of classical modulation theory, as summarised by Steps (I)-(II) in Section \ref{sec:introduction}, to the fast-slow setting. Fast-slow counterparts to Steps (I) and (II) were proposed in Sections \ref{sub:geometric_blow_up} and \ref{sub:method_of_multiple_scales} respectively, and formulated for the general class of PDE systems \eqref{eq:dynamic_general_form}. In order to achieve Step (I), we built upon recent findings in \cite{Jelbart2022} in order to show the classical multi-scale ansatz can be reformulated as a geometric blow-up transformation. In order to achieve Step (II), we coupled the geometric blow-up approach from Step (I) to a fast-slow extension of the classical asymptotic method of multiple scales.
		
		The methods developed in Section \ref{sec:geometric_blow-up_and_the_modulation_equations} can be used to systematically derive modulation equations which govern the dynamics of formal approximations near dynamic bifurcations. This was demonstrated concretely in Section \ref{sec:model_problems}, for the model problems (M1)-(M4) which were introduced in Section \ref{sec:dynamic_Turing_instability}. The modulation equations are given in and described by Theorems \ref{thm:modulation_equations_sh}, \ref{thm:modulation_equations_b}, \ref{thm:modulation_equations_ks} and \ref{thm:modulation_equations_ns}. In each case, we obtained equations with the same general form as their static counterpart (real GL, complex GL, coupled complex GL and Cahn-Hilliard respectively), however with some notable distinguishing features including (i) additional terms and time-dependent coefficients induced by the parameter drift $\dot \mu = \eps$; (ii) dependence on desingularized (as opposed to simply rescaled) space and time $\bar{\bm x}$ and $\bar t$, and (iii) the fact that they are posed in a non-trivial geometry (the blown-up space). Although we did not utilise it in the present work, it is also worthy to emphasise that the blow-up procedure also yielded improved hyperbolicity properties with regard to the equations in local coordinate charts $\mathcal K_l$; see again Remarks \ref{rem:hyperbolicity}, \ref{rem:hyperbolicity_2}, \ref{rem:hyperbolicity_3} and \ref{rem:hyperbolicity_4}. The presence of this important feature, which is well-known and fundamental to the utility and success of geometric blow-up analyses in finite-dimensional systems, provides significant motivation for the continued effort to develop these methods for PDE systems.
		
		\
		
		Finally, although we are confident that the contributions of this article provide a promising foundation for future work in this area, a substantial amount of work remains to be done. In particular, Step (III) must be addressed in the fast-slow setting, and, given that this is possible, detailed dynamical analyses of the modulation equations obtained in Steps (I)-(II) are necessary in order to infer something about the dynamics of the original problem. These non-trivial tasks and other related problems are left for future work.

		

		\bibliographystyle{siam}
		\bibliography{swift_hohenberg}
		
		\appendix
		
		\section{Proof of Lemma \ref{lem:L}}
		\label{app:proof_of_lemma_L}
		
		We need to apply the replacements in \eqref{eq:partial_derivatives} to the definition for $\Lop$ in \eqref{eq:L}. In particular, we need to express the right-hand side of
		\begin{equation}
			\label{eq:D_replacement}
			\Diff_x^\alpha = \frac{\partial^{|\alpha|}}{\partial x_1^{\alpha_1} \partial x_2^{\alpha_2} \cdots \partial x_n^{\alpha_n}} \mapsto 
			\mathcal D_x^\alpha = \prod_{k=1}^p (\partial_{x_k} + r \partial_{\bar x_k})^{\alpha_k} \prod_{l=p+1}^n \partial_{x_l}^{\alpha_l} 
		\end{equation}
		as a power series in $r$. We consider only the more difficult case with $p < n$. The same calculations can be applied in the case that $p = n$ after replacing the right-most product in \eqref{eq:D_replacement} by $1$.
		
		Using the binomial formula we have
		\[
		(\partial_{x_k} + r \partial_{\bar x_k})^{\alpha_k} =
		\sum_{q=0}^{\alpha_k} {\alpha_k \choose q} \partial_{x_k}^{\alpha_k - q} \partial_{\bar x_k}^q r^q =
		\sum_{q=0}^\infty C_{\alpha_k q_k}\left( \partial_{x_k}, \partial_{\bar x_k} \right) r^q ,
		\]
		where we assume that mixed partial derivatives commute (this is valid since we assume sufficient regularity), and the summands $C_{\alpha_k q_k}( \partial_{x_k}, \partial_{\bar x_k})$ are given by \eqref{eq:Calphak}. 
		Using the generalised Cauchy formula to evaluate the product leads to the following power series in $r$:
		\[
		\begin{split}
			\prod_{k=1}^p (\partial_{x_k} + r \partial_{\bar x_k})^{\alpha_k} &= \prod_{k=1}^p \sum_{q_k=0}^\infty 
			C_{\alpha_k q_k}\left( \partial_{x_k}, \partial_{\bar x_k} \right) r^{q_k} \\
			&= \sum_{l=0}^{\infty} \left( \sum_{\substack{q_1, q_2, \ldots , q_p \geq 0 \\ q_1 + q_2 + \cdots + q_p = l}} \prod_{k=1}^p 
			C_{\alpha_k q_k}\left( \partial_{x_k}, \partial_{\bar x_k} \right) \right) r^l .
		\end{split}
		\]
		Substituting this into the right-hand side of \eqref{eq:D_replacement} leads to
		\[
		\mathcal D_x^\alpha = \sum_{l=0}^{\infty} \mathcal D_x^{(\alpha,l)} r^l ,
		\]
		where
		\[
		\mathcal D_x^{(\alpha,l)} := \left( \sum_{\substack{q_1, q_2, \ldots , q_p \geq 0 \\ q_1 + q_2 + \cdots + q_p = l}} \prod_{k=1}^p 
		C_{\alpha_k q_k}\left( \partial_{x_k}, \partial_{\bar x_k} \right) \right) \prod_{s=p+1}^n \partial_{x_s}^{\alpha_s} .
		\]
		Note that this is actually a finite sum due to the defining equation for the $C_{\alpha_k q_k}$, which implies that only the first $m \in \mathbb N$ terms can be nonzero. This follows from the fact that $|\alpha| \leq m$, which implies that the right-hand side of \eqref{eq:D_replacement} terminates at $O(r^m)$. 
		
		Substituting this into
		\[
		\Lop_j = \sum_{|\alpha| \leq m} a_{j}^{(\alpha)}(\bm x) \Diff_{\bm x}^\alpha \mapsto
		\mathcal L_j = 
		\sum_{l=0}^\infty \mathcal L_j^{(l)} r^l
		\]
		shows that
		\[
		\mathcal L_j^{(l)} = \sum_{|\alpha| \leq m} a_{j,\alpha}(\bm x) \mathcal D_x^{(\alpha,l)} ,
		\]
		which is the same as equation \eqref{eq:mathcal_Lj}, as required.
		
		\
		
		The fact that $\mathcal L^{(0)} = \Lop$ (which depends on $\bm x$ but not $\bar{\bm x}$) can be shown directly using the formulas above, and the expression for $\mathcal L_j \psi_j$ in \eqref{eq:Lj_psij} is obtained directly using the Cauchy formula.
		\qed

	\end{document}